\documentclass[reqno,12pt,a4paper]{amsart}

\usepackage{amsmath,amsthm,amssymb,amsfonts,dsfont,ifpdf,mathtools,verbatim}
\usepackage{enumitem}
\usepackage{graphicx}
\usepackage[usenames,dvipsnames]{color}
\usepackage[all]{xy}

\usepackage[colorlinks=true, linkcolor=blue, urlcolor=blue, citecolor=ForestGreen]{hyperref}
\numberwithin{equation}{section}

\newtheorem {thm}    {Theorem}[section]
\newtheorem {lem}      [thm]    {Lemma}

\newtheorem {cor}  [thm]    {Corollary}
\newtheorem {prop}[thm]    {Proposition}
\newtheorem* {prop*} {Proposition}

\newtheorem*{claim*}   {Claim}

\newtheorem*{conj*} {Conjecture}

\theoremstyle{definition}
\newtheorem {defn} [thm]    {Definition}
\newtheorem {ex}    [thm]    {Example}
\newtheorem {rmk}    [thm]    {Remark}
\newtheorem*{rmk*}  {Remark}

\newtheorem*{qst*} {Question}

\newtheorem* {problem*}{Problem}

\newcounter{AbcT}

\numberwithin{equation}{section}

\newcommand {\supl}   {\sup\limits}
\newcommand {\infl}   {\inf\limits}
\newcommand {\liml}   {\lim\limits}
\newcommand {\limsupl}   {\limsup\limits}

\newcommand {\E} {{\mathbb E}}

\newcommand {\N} {{\mathbb N}}
\newcommand {\Q} {{\mathbb Q}}

\renewcommand {\P}  {{\mathbb P}}
\newcommand {\R} {{\mathbb R}}

\newcommand {\Z} {{\mathbb Z}}

\newcommand {\cB} {{\mathcal B}}

\newcommand {\cF} {{\mathcal F}}

\newcommand {\cT} {{\mathcal T}}

\DeclareMathOperator{\supp}{supp}

\DeclareMathOperator{\SL}{SL}
\DeclareMathOperator{\PSL}{PSL}

\newcommand{\eps}{\varepsilon}

\newcommand {\IGNORE}[1]  {}

\newcommand {\norm}[1] {\left\| {#1} \right\|}

\newcommand {\La} {{\Lambda}}
\newcommand{\Ga}{\Gamma}

\newcommand{\lb}{\liminf\limits_{n\to\infty}\frac{1}{n}\log{\mu_n}}
\newcommand{\ub}{\limsup\limits_{n\to\infty}\frac{1}{n}\log{\mu_n}}

\begin{document}

	\title[Large deviations for random walks on free products]{Large deviations for random walks on free products of finitely generated groups}
	\author[E.~Corso]{Emilio Corso}
	\address[E. C.]{ETH Z\"urich, R\"amistrasse 101
		CH-8092 Z\"urich
		Switzerland}
	\email{emilio.corso@math.ethz.ch}
	\date{\today}
	\keywords{Large deviations, random walks, free groups, free products, Gromov-hyperbolic groups, cone types, regular trees}
	
	\subjclass[2010]{60B15, 60F10, 60G50, 05C81}
	
	\begin{abstract}
		We prove existence of the large deviation principle, with a proper convex rate function, for the distribution of the renormalized distance from the origin of a random walk on a free product of finitely generated groups. As a consequence, we derive the same principle for nearest-neighbour random walks on regular trees.
	\end{abstract}
	\maketitle
	
	\tableofcontents

	\section{Introduction and main result}
	The study of random walks on algebraic and geometric structures, most notably graphs and groups, has attracted considerable attention over the last four decades. Initiated by Polya's celebrated results on recurrence and transience of symmetric simple random walks on integer lattices (\cite{Polya}), the subject rose to prominence in the sixties, starting with Kesten's foundational work in the context of groups (\cite{Kesten}). It was later repopularised, mainly owing to pioneering contributions due to Kaimanovich, R.~Lyons, Varopoulos, Vershik, to name but a few; several directions of investigation gradually emerged, alongside new connections with various branches of pure and applied mathematics. For further details, we refer the reader to Woess' monograph~\cite{Woess} and the extensive bibliography therein.
	
	In this article, we confine ourselves to the study of random walks on a class of finitely generated groups, and specifically to the investigation of the asymptotic properties of the distribution of the renormalized distance from the origin. Prior to stating our main result, we provide a brief overview of the context within which it can be inscribed. 
	
	Let $G$ be a finitely generated group, endowed with the discrete topology, and $\mu$ a probability measure on $G$. The measure $\mu$ defines a right random walk $(Y_{n})_{n\in \N}$ started at $Y_0=e$, the identity element of $G$, given by $Y_n=X_1\cdots X_n$ for every $n\geq1$, where the $X_n$'s are independent $G$-valued random variables identically distributed according to $\mu$ (see Section~\ref{preliminaries} for precise definitions). Select a subset $S\subset G$ generating the group $G$. It determines a length function $\ell$ on $G$, measuring the size of its elements with respect to $S$; more precisely, for every $g\in G$, $\ell(g)$ is the minimal number of elements from the set $S\cup S^{-1}$ which are needed to obtain $g$ by multiplying them together. This corresponds to the path distance from the identity on the Cayley graph of $G$ with respect to the generating set $S$. To simplify the discussion, and in accordance with the cases of utmost interest, we shall always assume that $S$ is finite, though this is not necessary for the validity of Theorem~\ref{main}, which represents the main contribution of the article.
	
	The following well-known result provides an analogue, in a possibly non-commutative setting, of the strong law of large numbers for sums of independent real random variables.
	
	\begin{thm}
		\label{SLLN}
		Assume that $\mu$ has finite first moment with respect to the length function $\ell$, that is $\int_{G}\ell(g)\;d\mu(g)<\infty$. Then, there exists a non-negative real number $\lambda$ such that 
		\begin{equation*}
		\lim\limits_{n\to\infty}\frac{1}{n}\;\ell(Y_n)=\lambda\quad \P\emph{-almost surely.}
		\end{equation*}
	\end{thm}     
	
	Theorem~\ref{SLLN} is a consequence of Kingman's subadditive ergodic theorem (\cite{Kingman}); for a proof, we refer to the original article of Guivarc'h~\cite{Guivarch}.
	
	The constant $\lambda$ appearing in Theorem~\ref{SLLN} is called the \emph{escape rate} (or \emph{speed}) of the random walk; it clearly depends on $\mu$ and on the length function $\ell$. 
	
	Once almost-sure convergence of the sequence $\bigl(\frac{1}{n}\ell(Y_n)\bigr)_{n\geq 1}$ is established, it is natural to enquire about the asymptotic behaviour of the deviations from the mean $\ell(Y_n)-n\lambda$. In this spirit, a central limit theorem was first established in~\cite{Sawyer-Steger} for the case of free groups; a second, more geometric proof of the same result was later provided by Ledrappier in~\cite{Ledrappier}. Subsequently, Bjorklund (\cite{Bjorklund}) transposed Ledrappier's argument to the setting of Gromov-hyperbolic groups (cf.~\cite{Gromov, Ghys-deLaHarpe}), proving a central limit theorem for the Green metric on the group $G$. The rationale behind the introduction of such a metric is of geometric nature: with respect to the Green metric, the horofunction boundary of $G$ is $G$-equivariantly homeomorphic to the Gromov boundary, a technical assumption which is instrumental in Bjorklund's approach. Thereafter, Benoist and Quint (\cite{Benoist-Quint}) extended the result to distance functions defined by word lengths, by adapting the method introduced in ~\cite{Benoist-Quint-linear}.
	
	\begin{thm}[{\cite[Thm.~1.1]{Benoist-Quint}}]
		\label{CLT}
		Let $G$ be a Gromov-hyperbolic group, and suppose that $\mu$ is a non-elementary and non-arithmetic probability measure on $G$ with finite second moment, that is $\int_{G}\ell(g)^{2}\;d\mu(g)<\infty$. Then the sequence of renormalized random variables
		\begin{equation*}
		\frac{1}{\sqrt{n}}(\ell(Y_n)-n\lambda)\;,\;n\geq 1,
		\end{equation*}
		converges in distribution to a non-degenerate Gaussian law. 
	\end{thm}
	For an explanation of the assumptions on the measure $\mu$ appearing in Theorem~\ref{CLT}, we refer the reader to~\cite{Benoist-Quint}.
	It is worth noticing that all earlier works on the central limit theorem in this context rely on the stronger assumption of finiteness of some exponential moment for $\mu$. A recent paper by Mathieu and Sisto (\cite{Mathieu-Sisto}), in which Theorem~\ref{CLT} is established for the yet broader class of acylindrically hyperbolic groups, also deserves mention.
	
	\smallskip
	In light of Theorem~\ref{SLLN}, it is clear that 
	\begin{equation}
	\label{rareevents}
	\P(|\ell(Y_n)-n\lambda|\geq \delta n)\overset{n\to\infty}{\longrightarrow}0 \text{ for any }\delta>0.
	\end{equation}
	We are interested in the decay rate of the probability of such rare events. Properly speaking, we ask whether the sequence of random variables $\bigl(\frac{1}{n}\ell(Y_n)\bigr)_{n\geq 1}$ satisfies the large deviation principle (see Section~\ref{largedeviations}); loosely, it amounts to asking if there is a well-defined exponential decay rate for the probability of events of the type appearing in ~\eqref{rareevents}. 
	
	It is natural to expect the large deviation principle to hold for a large class of finitely generated groups, in particular for Gromov-hyperbolic groups; we expand slightly more on possible extensions of our approach\footnote{After the first version of this paper appeared, Boulanger, Mathieu, Sert and Sisto~\cite{Boulanger-Mathieu-Sert-Sisto} proved existence of the large deviation principle for random walks on geodesic hyperbolic spaces, thus encompassing the case of walks on Gromov-hyperbolic groups. The underpinning strategy does not differ substantially from our approach, though it relies on deeper geometric considerations.} in this direction in Section~\ref{final}. The applicability of the same strategy to such extensions, as well as to analogous questions in random matrix products, is already mentioned in~\cite{Cagri-note}.
	
	Our main result establishes the existence of the large deviation principle, with a proper convex rate function, for the collection of non-trivial free products of finitely generated groups, under a non-degeneracy assumption on the semigroup $\Ga$ generated by the support of the driving measure $\mu$. Specifically, we require that $\Ga$ is \emph{pattern-avoiding}: there exists a positive integer $D>0$ such that, for any reduced word $\omega=y_1\cdots y_D$ of \emph{type size} $D$ in the free product, there is an element $g\in \Gamma\setminus\{e\}$ which neither \emph{starts with} $\omega$ nor \emph{ends with} $\omega^{-1}$. For a precise definition, we refer to Section~\ref{patternavoid}, while the relevance of this condition to the purposes of the proof is explained in Section~\ref{strategy}. For the sake of illustration, we hasten to observe that the pattern-avoidance condition is fulfilled, for instance, if $\Ga$ intersects two distinct factors of the free product non-trivially (see Example~\ref{pattavoid}). 
	
	Expanding upon the latter observation, we precede the statement of the main result, Theorem~\ref{main}, with a simpler and more concise version which already singles out a broad class of admissible driving measures. 
	
	\begin{prop}
		\label{particularcase}
		Let $r\geq 2$ be an integer, $G_1,\dots,G_r$ non-trivial finitely generated groups, $G=G_1\ast\cdots \ast G_r$ their free product, $S_i$ a finite generating set of $G_i$ for $i=1,\dots,r$, $S=\bigcup_{i=1}^{r}S_i$, $\ell$ the length function on $G$ determined by $S$. Let $\mu$ be a probability measure on $G$, and assume its support generates a semigroup $\Ga$ with the property that, for any $i\in \{1,\dots,r \}$, there is an element $g\in \Ga$ which neither starts nor ends in the factor $G_i$.  If $(Y_n)_{n\geq 0}$ is a right random walk on $G$ with increments distributed according to $\mu$, then the sequence of random variables 
		$\bigl(\frac{1}{n}\ell(Y_n)\bigr)_{n\geq 1}$
		satisfies the weak large deviation principle with a convex rate function.
	\end{prop}
	
	Observe that any semigroup $\Ga$ fulfilling the assumptions of Proposition~\ref{particularcase} avoids patterns of type size $D=1$ (the converse clearly fails, as shown in Example~\ref{pattavoid}). In order to deal with more general pattern-avoiding semigroups, our method compels us to impose an additional constraint on the size of the factors $G_1,\dots,G_r$.
	
	The complete formulation of our results reads as follows:
	
	\begin{thm}
		\label{main}
		Let $r\geq 2$ be an integer, $G_1,\dots,G_r$ non-trivial finitely generated groups of subexponential growth, $G=G_1\ast\cdots \ast G_r$ their free product, $S_i$ a finite generating set of $G_i$ for $i=1,\dots,r$, $S=\bigcup_{i=1}^{r}S_i$, $\ell$ the length function on $G$ determined by $S$. Suppose that $\mu$ is a probability measure on $G$ whose support generates a pattern-avoiding semigroup, and let $(Y_n)_{n\geq 0}$ be a right random walk on $G$ with increments distributed according to $\mu$. 
		\begin{enumerate}
			\item The sequence of random variables 
			$\bigl(\frac{1}{n}\ell(Y_n)\bigr)_{n\geq 1}$
			satisfies the weak large deviation principle with a convex rate function $I\colon \R_{\geq 0}\to [0,\infty]$.
			\item If $\mu$ has a finite exponential moment, then $I$ is a proper function and the sequence $\bigl(\frac{1}{n}\ell(Y_n)\bigr)_{n\geq 1}$ satisfies the full large deviation principle with rate function $I$.
			\item If $\mu$ has finite moment-generating function, then $I$ is the Fenchel-Legendre transform of the limiting logarithmic moment generating function of the sequence $\bigl(\frac{1}{n}\ell(Y_{n})\bigr)_{n\geq 1}$.
		\end{enumerate}
	\end{thm} 
	
	A close inspection of the proof of Lemma~\ref{weaklengthadd} reveals that the whole argument leading to Theorem~\ref{main} can be readily adapted to establish Proposition~\ref{particularcase}. In particular, the last two assertions of Theorem~\ref{main} remain equally valid in the setting of Proposition~\ref{particularcase}.
	
	For a precise definition of all the terms involved in the statement of Theorem~\ref{main}, we refer the reader to Sections~\ref{preliminaries} and~\ref{largedeviations}. Let us just recall here that a probability measure $\mu$ on $G$ is said to have \emph{a finite exponential moment} if $\int_{G}\exp{(\tau\ell(g))}\;d\mu(g)<\infty$ for some $\tau>0$, and that we say it has \emph{finite moment-generating function} if $\int_{G}\exp{(\tau\ell(g))}\;d\mu(g)<\infty$ for every $\tau>0$.
	
	
	By taking $G_i=\Z$ for all $i=1,\dots,r$, we settle in particular the question of existence of the large deviation principle for random walks on free groups; in turn, this yields the result for nearest-neighbour random walks on locally finite regular trees (a straightfoward adaptation of the proof of Theorem ~\ref{main} allows to deal with regular trees of odd degree as well). For the sake of simplicity, we state the corollary in the case relevant for applications to (possibly lazy\footnote{A $G$-random walk $(Y_n)_{n}$ is customarily called \emph{lazy} if $\mu(e)\geq 1/2$; here, for convenience, we employ the terminology in order to refer to the more general case $\mu(e)>0$.}) simple random walks on trees.
	
	\begin{cor}
		\label{cor}
		Let $G$ be a free group on $r\geq 1$ generators, and let $S$ 
		be a free set of generators. Assume $\mu$ is a probability measure on $G$ whose support is contained in $S\cup S^{-1}\cup\{e\}$, and let $(Y_n)_{n\geq 0}$ be a right random walk on $G$ with increments distributed according to $\mu$. The sequence of random variables $\bigl(\frac{1}{n}\ell(Y_n)\bigr)_{n\geq 1}$, where $\ell$ is the length function on $G$ determined by $S$, satisfies the large deviation principle with a proper, convex rate function, coinciding with the Fenchel-Legendre transform of the limiting logarithmic moment generating function of $\bigl(\frac{1}{n}\ell(Y_n)\bigr)_{n\geq 1}$. 
	\end{cor}
	Notice that the case $r=1$ of Corollary~\ref{cor} is not covered in principle by Theorem~\ref{main}; on the other hand, this case is a well-known, elementary instance of Cramer's theorem (cf.~\cite[Thm.~2.2.3]{De-Ze}) on deviations of the empirical mean of independent, identically distributed real random variables. Incidentally, our method would be readily applicable to this case as well, as we point out in section~\ref{final}, thus yielding an indirect proof of Cramer's theorem for simple random walks on $\Z$ (and $\Z^d$).
	
	\begin{rmk}
		Several remarks about Theorem~\ref{main} and Corollary~\ref{cor} are in order. 
		\begin{enumerate}
			\item A version of Grushko's theorem (\cite{Lyndon}) asserts that every finitely generated group can be decomposed in an essentially unique way as a free product of finitely many groups, which are not further decomposable as non-trivial free products. Notwithstanding this structural result, the class of examples Theorem~\ref{main} deals with is restricted, because of the limitations imposed on the generating set $S$, whose peculiar structure is crucial to our approach (cf.~Section~\ref{strategy} below). 
			
			On the other hand, the pattern-avoiding assumption on the semigroup $\Ga$ is by no means necessary for the result to hold; it is only a convenient manner of identifying a large class of examples to which our method applies\footnote{It becomes clear from the proofs that the very same method takes care, in addition, of some cases such as $\supp{\mu}\subset\{(ab)^{n}:n\in \Z\}$ in $G=\langle a,b\rangle$ a free group on two generators, in which the semigroup generated by $\supp{\mu}$ is not pattern-avoiding. Ruling out such trivial examples, it doesn't seem unlikely that a failure of the pattern-avoidance condition actually forces a conjugate of $\Ga$ to lie in one of the factors.}. Therefore, it stands to reason to expect that a technical refinement of our method would allow to weaken the assumption on the support of the driving measure, and deal with the case in which no conjugate of the semigroup $\Ga$ lies in a single factor. In this respect, see the proof of Lemma~\ref{weaklengthadd}. A similar result would yield, notably, that existence of the LDP for the length function is stable under taking free products.
			\item The result in Corollary~\ref{cor} might also be derived, when $2r=p+1$ for a positive prime $p$, from the large deviation principle for random walks on linear algebraic groups\footnote{This has been pointed out to the author by C.~Sert.} (see~\cite[Thm.~3.3]{Cagri-annals}), by choosing an appropriate representation of the free group in the projective special linear group $\PSL_2(\Q_p)$\footnote{The rank-one algebraic group $\PSL_2(\Q_p)$ acts by isometries on its Bruhat-Tits tree $\mathcal{T}$, which is regular of degree $p+1$ (for the construction, we refer to Serre's book~\cite{Serre}). Hyperbolic elements of $\PSL_2(\Q_p)$ act on $\mathcal{T}$ as hyperbolic elements in the geometric sense (cf.~\cite[Sec.~6]{Quint}). Choosing a base vertex $o\in \mathcal{T}$, the translation distance from $o$ corresponds, up to a multiplicative factor, to the operator norm on $\PSL_2(\Q_p)$ derived from a choice of a $K$-invariant ultrametric norm on the local field $\Q_p$, where $K<\PSL_2(\Q_p)$ is the compact stabilizer of $o$. Selecting hyperbolic elements which generate a Zariski-dense free subgroup of $\PSL_2(\Q_p)$ amounts to definining an isometric embedding of the corresponding free group in $\PSL_2(\Q_p)$.}.
			Our approach is different in that it resorts to the intrinsic geometric properties of the free group, rather than appealing to a representation. 
			\item Sharp large deviations estimates for the word-length functional of finite-range random walks on free groups are already present in the work of Lalley\footnote{We thank S.~M\"{u}ller for drawing our attention to this reference.} (\cite[Thm.~7.2]{Lalley}). The techniques adopted there differ significantly from ours, hinging on an extension of the Perron-Frobenius theory of nonnegative matrices to certain inhomogeneous matrix products; they yield finer information on the rate function, notably strict convexity, but require the assumption of \emph{aperiodicity} of the random walk (cf.~\cite{Lalley}), which our method does not necessitate.
		\end{enumerate}
	\end{rmk}
	
	\begin{rmk}
		Our hypothesis on the support of $\mu$ is unrelated to the choice of the generating set $S$. This makes Theorem~\ref{main} applicable, for instance, to the following circumstance, in which the driving measure has  \emph{apriori} no connection with the generating set. Let $G$ be a finitely generated group, $H<G$ a finite-index subgroup (hence $H$ is finitely generated by Schreier's subgroup lemma), $S\subset H$ a finite generating set of $H$, $T\subset G$ a set of representatives of right cosets of $H$ in $G$, $\tilde{S}=\{st:s\in S,t\in T\}$ the corresponding finite generating set of $G$. Suppose that $\tilde{\mu}$ is a probability measure on $G$ whose support is contained $\tilde{S}\cup \tilde{S}^{-1}\cup \{e\}$, thus giving rise to a nearest-neighbour random walk $(Y_n)_{n\in \N}$ on the Cayley graph $\text{Cay}(G,\tilde{S})$ of $G$ with respect to $\tilde{S}$. Let $\tau_{1}<\tau_2<\cdots \tau_n<\cdots$ be the strictly increasing sequence of stopping times defined by the successive instants in which the random walk visits $H$; they are all finite $\P$-almost surely, since $H$ has finite index in $G$. By an iterative application of the strong Markov property (\cite[Chap.~17]{Klenke}) to the process $(Y_{n})_{n\in \N}$, it follows that $H$-valued process $(Y_{\tau_n})_{n\in \N}$ (where we agree that $Y_{\tau_0}=e$) is a right random walk on $H$ driven by a measure $\mu$ having finite moment-generating function with respect to the word length determined by $S$; if $H$ is a non-trivial free product of finitely generated groups, all conclusions of Theorem~\ref{main} hold.
		
		An example of interest is the arithmetic group $\SL_2(\Z)$, which contains a multitude of finite-index free subgroups (cf.~\cite[Chap.~II]{deLaHarpe}).  
	\end{rmk}
	
	\subsection{Outline of the strategy}
	\label{strategy} 
	
	To illustrate the overarching strategy of our proof of Theorem~\ref{main}, it is informative to recall the indirect approach to the proof of Cramer's theorem for i.i.d.~real random variables, put forward by Lanford (\cite{Lanford}). If $(X_n)_{n\geq 1}$ is a sequence of i.i.d.~$\R$-valued random variables and $S_{n}=\sum_{i=1}^{n}X_i$ denotes the sequence of partial sums, then, for every $x\in \R$ and $\eps>0$, the limit $\lim_n\frac{1}{n}\log{\P\bigr(\frac{1}{n}S_n\in (x-\eps,x+\eps)\bigr)}$ exists in $[-\infty,0]$ by supermultiplicativity of the sequence $\P\bigl(\frac{1}{n}S_{n}\in (x-\eps,x+\eps) \bigr)$, which in turn is given by additivity of the the process $(S_n)_{n\geq 1}$. The weak LDP now follows from a standard result in large deviations' theory (see Proposition~\ref{criterion}). Similarly, the weak LDP holds for any additive functional\footnote{This means that $\ell'(Y_{n+m})=\ell'(Y_n)+\ell'(Y_m)$ for every $n,m\geq 0$.} $\ell'$ of a random walk $(Y_n)_{n\in \N}$ on a group $G$.
	
	The major obstacle, when attempting to transport this argument to our context, lies in the defect of additivity of length functions on discrete groups; subadditivity only ensures supermultiplicativity of the sequence $\P(\frac{1}{n}\ell(Y_n)\in I)$ for intervals of the form $I=(-\infty,x)$. Still, if the random walk can be restricted to subsets in which the length function is almost additive (cf.~Lemma~\ref{weaklengthadd} and the terminology introduced thereunder) without sizeable loss in the exponential decay rate of the corresponding probabilities, then Lanford's approach carries over almost unaffectedly. 
	Specifically, the structure of the generating set $S$, obtained by concatenating generating sets of the various factors, enables us to quantify neatly the lack of additivity in terms of the reduced-word expansion of the elements involved; the pattern-avoiding assumption on the semigroup $\Ga$ can then be leveraged to confine the attention to subsets on which the length function is \emph{weakly additive}, and which are attained by the random walk with sufficiently high probability on an exponential scale. This is detailed in Lemma~\ref{weaklengthadd}. Once a uniform lower bound for the loss of additivity is achieved, it is possible to deduce that, if $\gamma=\limsup_{n}\frac{1}{n}\log{\P\bigl(\frac{1}{n}\ell(Y_n)\in (x-a,x+a)\bigr)}$ for given $x,a\in \R_{>0}$, then the bound\linebreak $\P\bigl(\frac{1}{n_k}\ell(Y_{n_k})\in (x-a,x+a)\bigr)\geq e^{n_k(\gamma-\eta)}$ ($\eta$ being an arbitrarily small parameter) holds for a non-lacunary sequence of integers $(n_k)_k$. The arithmetic nature of such a sequence permits to deduce a lower bound $\liminf_{n}\frac{1}{n}\log{\P\bigl(\frac{1}{n}\ell(Y_n)\in (x-b,x+b)\bigr)}\geq \gamma-\eta$, at the minor cost of choosing $b$ strictly larger than $a$; this is the purpose of Lemma~\ref{lowerbound}, which in a sense plays the role of Fekete's lemma in Lanford's original argument. The actual proof of the weak LDP is then articulated in Proposition~\ref{existence}. Convexity of the rate function follows from entirely analogous arguments, as the proof of Proposition~\ref{convex} highlights. Finally, the remaining properties of the rate function mentioned in the statement of Theorem~\ref{main} are inferred from well-known foundational results in the theory of large deviations (cf.~Proposition ~\ref{weakstrong}, Theorem~\ref{momentgen} and Sections~\ref{further},~\ref{FL}). 
	
	As a concluding comment, let us point out that the strategy outlined here parallels
	arguments employed in~\cite{Cagri-annals} to deal with large deviations of the Cartan projection of random matrix products; in this context, a weak form of additivity for the Cartan projection is satisfied on $(r,\eps)$-Schottky semigroups, as shown by Benoist (\cite{Benoist}). The restriction of the random walk to such semigroups is then made possible by a result of Abels-Margulis-Soifer (\cite{Abels-Margulis-Soifer}), establishing the ubiquity of $(r,\eps)$-proximal elements in Zariski-dense semigroups.  
	
	\subsection{Outline of the article}
	
	We begin with some preliminaries on random walks on finitely generated groups in Section~\ref{preliminaries}, which mainly serve the purpose of fixing notation and elucidating the nature of the pattern-avoiding assumption we impose on the semigroup $\Ga$. In Section~\ref{largedeviations} we recall some standard terminology from the theory of large deviations, together with a few general facts which are employed in the proof of Theorem~\ref{main}. Sections~\ref{exist} and~\ref{conv} are devoted to the proof our main result~\ref{main}; specifically, in Section~\ref{exist} we establish existence of the large deviation principle, while in Section~\ref{conv} we prove convexity of the rate function, which, together with properness, allows us to identify it as the convex conjugate of a logarithmic moment generating function. Finally, in Section~\ref{final} we assemble ideas on possible generalizations of Theorem~\ref{main}, list some open questions and formulate related conjectures. 
	
	
	\subsection*{Acknowledgments}
		This work owes a major debt to \c{C}agri Sert, to whom the author expresses his gratitude for several insightful comments and enlightening conversations. Special thanks go to the referee for a thorough reading of the article, which tremendously helped improve its quality. Lastly, we would like to thank Manfred Einsiedler for valuable remarks on a preliminary version, as well as Sebastian M\"{u}ller for providing many useful references and observations.

	\section{Random walks on groups}
	\label{preliminaries}
	\subsection{Word length and metric on a finitely generated group}
	Convenient sources for the material presented hereunder are~\cite{deLaHarpe,Lyons-Peres,Woess}. 
	
	Let $G$ be a finitely generated group with identity element $e$, $S\subset G$ a finite generating set. Let $S^{-1}=\{s^{-1}:s\in S\}$ denote the set of inverses of the elements in $S$, so that  
	\begin{equation*}
	G=\{s_1\cdots s_n:n\geq 1,s_i \in S\cup S^{-1} \text{ for all }1\leq i\leq n\}.
	\end{equation*} 
	
	We define the \emph{word length} $\ell$ detemined by the generating set $S$ as the function $\ell\colon G \to \N$ given by 
	\begin{equation*}
	\ell(g)=\inf\{n\in \N:\text{ there exist }s_1,\dots, s_n\in S\cup S^{-1} \text{ such that }g=s_1\cdots s_n  \}
	\end{equation*}
	for every $g\in G$,
	with the understanding that $\ell(e)=0$. Then $\ell$ is a length function, meaning that it satisfies the following properties: 
	\begin{itemize}
		\item $\ell(g)\geq 0$ for all $g\in G$ and $\ell(g)=0$ if and only if $g=e$; 
		\item $\ell(g^{-1})=\ell(g)$ for all $g \in G$;
		\item $\ell(g_1g_2)\leq \ell(g_1)+\ell(g_2)$ for all $g_1,g_2 \in G$.
	\end{itemize}
	
	The word length $\ell$ determines a distance function $d$ on $G$, called the \emph{word metric} associated to the generating set $S$, defined by $d(g_1,g_2)=\ell(g_1^{-1}g_2)$ for all $g_1,g_2\in G$. The word metric $d$ is invariant for the action of $G$ on itself by left translation, namely $d(gg_1,gg_2)=d(g_1,g_2)$ for all $g,g_1,g_2\in G$.
	
	\smallskip
	We denote by $\text{Cay}(G,S)=(V,E)$ the Cayley graph of $G$ with respect to $S$; we recall that this is the simple, undirected graph whose vertex set $V$ is the group $G$, where two vertices $g_1,g_2\in V$ are connected by an edge $e=\{g_1,g_2\}\in E$ if and only if $d(g_1,g_2)=1$. In other words, there is an edge connecting $g_1$ to $g_2$ if and only if there is $s\in S\cup S^{-1}\setminus\{e\}$ such that $g_2=g_1 s$. The graph $\text{Cay}(G,S)$ is connected, transitive and locally finite of degree $|S\cup S^{-1}\setminus \{e\}|$. The word metric $d$ on $G$ corresponds, via this identification, to the path distance on the vertex set $V$ (cf.~\cite[Chap.~3]{Lyons-Peres}).
	
	\smallskip
	Let $B^{G}(T)=\{g\in G:\ell(g)\leq T \}$ be the closed $d$-ball of radius $T$ centered at the identity, for any $T\in \R_{\geq 0}$. As the sequence $\bigl(|B^{G}(n)|\bigr)_{n\geq 1}$ is submultiplicative, the limit $\gamma_{S}=\lim_n|B^{G}(n)|^{1/n}$ exists; we say that $G$ has \emph{subexponential growth} if $\gamma_S=1$, a property which is actually independent of the generating set $S$. Recall that a broad class of finitely generated groups with subexponential (in fact, polynomial) growth consists of nilpotent groups (\cite{Wolf}). 
	
	\smallskip
	If $G=G_1\ast \cdots \ast G_r$ is the free product (cf.~\cite{Bourbaki-algebra}) of finitely generated groups $G_1,\dots,G_r$, we shall always restrict our considerations to the following kind of generating sets (and corresponding word lengths): we fix generating sets $S_i\subset G_i$ for each factor $G_i$ of the free product, and take the union $S=\bigcup_{i=1}^{r}S$ as generating set for $G$. 
	
	\subsection{Free products and pattern-avoiding subsets}
	\label{patternavoid}
	
	Let $r\geq 2$ be an integer, $G_1,\dots,G_r$ non-trivial finitely generated groups, and let
	$G=G_1\ast \cdots \ast G_r$ be the free product of the $G_i$'s. We shall identify each $G_i,\;1\leq i\leq r$, with its isomorphic copy embedded in $G$. 
	
	\begin{lem}[{\cite[Chap.~II, Prop.~1]{deLaHarpe}}]
		\label{uniquedec}
		For any non-trivial element $g\in G$, there exist uniquely determined non-trivial elements $x_1\in G_{i_1},\dots,x_m\in G_{i_m}$, with $i_j\neq i_{j+1}$ for all $1\leq j\leq m-1$, such that $g=x_1x_2\cdots x_m$.
	\end{lem}
	
	Any product $x_{1}\cdots x_m$ as in Lemma~\ref{uniquedec} is referred to as a \emph{reduced word} of \emph{type size} $m$ in the free product; correspondingly, we shall also say that $g=x_1\cdots x_m$ is an element of type size $m$. For any $i\in \{1,\dots,m\}$, we call the element $x_i$ the $i$-th \emph{letter} of the reduced word $x_1\cdots x_m$.
	
	\begin{rmk}
		Suppose that we fix a generating set $S_i\subset G_i$ for each factor of the free product, and let $\ell_i$ denote the associated word length on $G_i$. Then, if $\ell$ is the word length determined by the generating set $S=\bigcup_{i=1}^{r}S_i\subset G$ and if $g,x_1,\dots,x_m$ are as in Lemma~\ref{uniquedec}, it holds $\ell(g)=\ell(x_1)+\cdots +\ell(x_m)$. Observe in particular that, while the \emph{word length} of an element $g\in G$ depends on the choice of the generating sets for the factors, the \emph{type size} of $g$ does not.  
	\end{rmk}
	
	Let $\omega=y_1\cdots y_d$ be a reduced word of type size $d$, $g\in G$ an element of type size at least $2$, with reduced-word decomposition $g=x_1\cdots x_m$. We shall say that $g$
	\begin{itemize}
		\item \emph{starts with} $\omega$ if $x_1\cdots x_{\inf\{d,\lfloor m/2\rfloor\}}=y_1\cdots y_{\inf\{d,\lfloor m/2\rfloor\}}$, and
		\item \emph{ends with} $\omega$ if $ x_{m-\inf\{d,\lfloor m/2\rfloor\}+1}\cdots x_m=y_1\cdots y_{\inf\{d,\lfloor m/2\rfloor\}}$,
	\end{itemize}
	where $\lfloor a \rfloor$ indicates the integer part of a real number $a$.
	Notice that the definition is independent of any choice of generating sets for the factors $G_1,\dots,G_r$ of the free product.
	
	\begin{ex}
		If $G=\langle a,b\rangle$ is a free group on two generators $a$ and $b$, then the element $abab$ starts with $ab$ and ends with $ab$, while the element $abab^{-1}a^{-1}$ starts with $ab$ and ends with $b^{-1}a^{-1}$. Also, according to our definition, the latter element starts with any word $ab\omega'$ obtained by juxtaposing a reduced word $\omega'$ to $ab$ in such a way that $ab\omega'$ is again a reduced word. 
	\end{ex} 
	
	A subset $\mathcal{T}\subset G$ is called \emph{pattern-avoiding} if there exists a positive integer $D>0$ such that, for any reduced word $\omega=y_1\cdots y_D$ of type size $D$ in the free product, there exists $g\in \mathcal{T}$ such that $g$ does not start with $\omega$ and does not end with $\omega^{-1}=y_{D}^{-1}\cdots y_1^{-1}$ (in particular, $g$ has type size at least $2$). In case we need to keep track of the integer $D$, we shall say that $\mathcal{T}$ \emph{avoids patterns of type size $D$}. The examples presented below clarify the notion.
	
	\begin{ex}
		\label{pattavoid}
		\begin{enumerate}
			\item Let $G=\langle a,b,c \rangle$ be a free group on three generators $a,b$ and $c$. The sets
			\begin{equation*}
			\mathcal{T}_1=\{ab,bc \},\; \mathcal{T}_2=\{acb,a^{3}bca^{-2}\},\; \mathcal{T}_3=\{aba^{-1},bab^{-1}\}
			\end{equation*}
			are pattern-avoiding, while the set
			\begin{equation*}
			\mathcal{T}_4=\{ab,ac^{2},ca^{-1} \}
			\end{equation*}
			is not pattern-avoiding, as all its elements either start with $a$ or end with $a^{-1}$.
			\item If $\mathcal{S}\subset g G_i g^{-1}$ for some $i\in\{1,\dots,r\}$ and some $g=x_1\cdots x_m\in G$, then the semigroup $\Ga$ generated by $\mathcal{S}$ is not pattern-avoiding: all its elements start with $x_1\cdots x_m$ and end with $(x_1\cdots x_m)^{-1}$.
			\item Suppose that there are indices $i\neq j\in \{1,\dots,r\}$ such that $\mathcal{S}\cap (G_i\setminus\{e\})\neq \emptyset$ and $\mathcal{S}\cap (G_j\setminus\{e\})\neq \emptyset$. Then the semigroup $\La$ generated by $\mathcal{S}$ is pattern-avoiding: if $x\in \mathcal{S}\cap (G_i\setminus\{e\})$ and $y\in \mathcal{S}\cap (G_j\setminus\{e\})$, then $\{xy,yx\}$ is pattern-avoiding and contained in $\La$.
			\item The semigroup generated by $\{aba,a^{2}ba^{2}\}$ in $G=\langle a,b\rangle$ avoids patterns of type size $1$, but does not satisfy the hypotheses of Proposition~\ref{particularcase}: its elements start and end in the factor $\langle a\rangle $. 
		\end{enumerate}
	\end{ex}
	
	Obviously, if $\mathcal{T}'\subset \mathcal{T}\subset G$ and $\mathcal{T}'$ is pattern-avoiding, then so is $\mathcal{T}$. Conversely, the following elementary observation is essential for our line of reasoning in Section~\ref{exist}: if $\mathcal{T}$ is pattern-avoiding, then there exists a finite subset $\mathcal{T}'\subset \mathcal{T}$ which is also pattern-avoiding\footnote{A simple enumeration of all possibilities shows that $\mathcal{T}'$ can be chosen with cardinality at most $3$.}. 
	
	\subsection{Random walks on finitely generated groups}
	
	Let $\mu$ be a probability measure on the group $G$; equivalently, $\mu$ is a function defined on $G$ taking non-negative real values and satisfying $\sum_{g\in G}\mu(g)=1$. Then $\mu$ defines a right random walk on $G$ as follows: let $(X_n)_{n\geq 1}$ be a sequence of independent, identically distributed $G$-valued random variables with common law $\mu$. Implicitly, we consider them to be defined over a probability space $(\Omega,\cF,\P)$, which will be fixed hereinafter. We define a $G$-valued stochastic process $(Y_n)_{n\in  \N}$ by setting $Y_0=e$, $Y_n=X_1\cdots X_n$ for every integer $n\geq 1$. The process $(Y_n)_{n\in \N}$ is called a \emph{right random walk} on $G$, issued from the origin $e$ with increments distributed according to $\mu$. Equivalently, one may defined the process $(Y_n)_{n\in \N}$ as a Markov chain on $G$ issued from $e$ with transition matrix $Q=(q(x,y))_{x,y \in G}$ given by $q(x,y)=\mu(x^{-1}y)$ for all $x,y \in G$ (cf.~\cite[Sec.~1.1]{Woess}). 
	
	\smallskip
	Let $\supp{\mu}=\{g\in G:\mu(g)>0\}$ be the support of the measure $\mu$. If $\supp{\mu}\subset S\cup S^{-1}$,\linebreak then the process $(Y_n)_{n\in \N}$ can also be interpreted as a nearest-neighbour random walk on the Cayley graph $\text{Cay}(G,S)$, where the walker in position $x$ moves to $xs$ with probability $\mu(s)$, for all $s \in S\cup S^{-1},x \in G$. Notice that we are not excluding the case $\mu(e)>0$, so that the walker may have positive probability of remaining where it is. 
	
	
	Let $\E[X]$ denote the expectation of a random variable $X\colon \Omega \to \R$ with respect to the probability measure $\P$. If $\mu$ has finite first moment, the sequence of renormalized averaged lengths
	\begin{equation*}
	\frac{1}{n}\E[\ell(Y_n)],\;n\geq 1,
	\end{equation*}
	is a subadditive real sequence, and as such converges to a limit $\lambda\in \R_{\geq 0}$, called the escape rate or speed of the random walk $(Y_n)_{n\in \N}$. As mentioned in the introduction (Theorem~\ref{SLLN}), $\P$-almost every trajectory $(y_n)_{n\geq 0}\in G^{\N}$ of the random walk actually satisfies $\frac{1}{n}\ell(y_n)\overset{n\to\infty}{\longrightarrow} \lambda$.
	
	\begin{rmk}
		\begin{enumerate}
			\item We could equally well consider random walks issued at any initial vertex $g_0\in G$, by defining $Y'_0=g_0$, $Y'_n=g_0X_1\cdots X_n$ for any $n\geq 1$. It is then natural to consider the renormalized distance $\frac{1}{n}d(g_0,Y'_n)$ which, by invariance of $d$ under left translations, equals precisely $\frac{1}{n}d(e,X_1\cdots X_n)=\frac{1}{n}\ell(Y_n)$. Hence, for the purpose of our considerations, there is no loss of generality in assuming that the random walk starts at the origin.
			\item Similarly, restricting to \emph{right} random walks does not result in any loss of generality; if $Y'_n=X_n\cdots X_1,\;n\geq 1,$ is a left random walk issued from the origin with driving measure $\mu$, then $(Y_n^{-1})_{n\in \N}$ is a right random walk with driving measure $\iota_{*}\mu$, given by $\iota_{*}\mu(g)=\mu(g^{-1})$ for every $g\in G$, and $\ell(Y_n^{-1})=\ell(Y_n)$ for every $n\in \N$.
		\end{enumerate}
	\end{rmk}

	\section{Large deviation principle}
	\label{largedeviations}
	In this section, we briefly review some of the terminology that is usually employed in the theory of large deviations. For a comprehensive introduction to the subject, the reader is referred to~\cite{De-Ze}.
	
	\smallskip
	Throughout this section, $X$ denotes a Hausdorff regular topological space, endowed with the Borel $\sigma$-algebra $\cB$. Let $(\mu_n)_{n\geq 1}$ be a sequence of Borel probability measures on $X$, $I\colon X \to[0,\infty]$ a lower semicontinuous function. The \emph{effective domain} of $I$ is the set $D_{I}=\{x\in X:I(x)<\infty\}$.
	
	\begin{defn}
		\label{LDP}
		We say that the sequence $(\mu_n)_{n\geq 1}$ satisfies the large deviation principle (or, in abridged form, LDP) with rate function $I$ if, for any Borel measurable set $\La\subset X$, 
		\begin{equation*}
		-\infl_{x\in \La^{\circ}}I(x)\leq \lb(\La)\leq \ub(\La)\leq-\infl_{x \in \overline{\La}}I(x)\;,
		\end{equation*}
		where $\La^{\circ}$ and $\overline{\La}$ denote the interior and the closure of $\La$, respectively.
	\end{defn}
	
	We observe in passing that, for a given sequence $(\mu_n)_{n\geq 1}$, there is at most one lower semicontinuous function $I$ for which the LDP can hold  (\cite[Lem.~4.1.4]{De-Ze}).
	
	In Definition~\ref{LDP}, it is obviously equivalent to require that
	\begin{equation}
	\label{lb}
	\lb(V)\geq -\infl_{x\in V }I(x) \text{ for every open set }V\subset X
	\end{equation}
	and 
	\begin{equation}
	\label{ub}
	\ub(F)\leq -\infl_{x\in F }I(x) \text{ for every closed set }F\subset X.
	\end{equation}
	If the lower bound~\eqref{lb} holds for any open set $V\subset X$, while the upper bound~\eqref{ub} holds just for all compact sets $K\subset X$, then we say that the sequence $(\mu_n)_{n\geq 1}$ satisfies the weak large deviation principle (weak LDP) with rate function $I$. 
	
	If $(Z_n)_{n\geq 1}$ is a sequence of $X$-valued random variables, and $\mu_n$ denotes the law of $Z_n$ for every $n\geq 1$, we shall say that $(Z_n)_{n\geq 1}$ satisfies the (weak) LDP if the sequence $(\mu_n)_{n\geq 1}$ satisfies the (weak) LDP. 
	
	\smallskip
	Under certain conditions, we may retrieve the full LDP from the existence of the weak LDP. The most common of these conditions involves the notion of exponential tightness. 
	
	\begin{defn}
		We say that a sequence $(\mu_n)_{n\geq 1}$ of Borel probability measures on $X$ is exponentially tight if, for every $\alpha\in \R_{\geq 0}$, there exists a compact set $K\subset X$ such that 
		\begin{equation*}
		\ub(X\setminus K)< -\alpha\;.
		\end{equation*}
	\end{defn}
	
	In other words, the mass is concentrated on compact sets, on an exponential scale.
	
	It is intuitively clear that exponential tightness enables to pass from a weak form of the LDP to a strong form, something which we clarify in the following proposition (cf.~\cite[Lem.~1.2.18]{De-Ze}).
	\begin{prop}
		\label{weakstrong}
		Let $(\mu_n)_{n\geq 1}$ be an exponentially tight sequence of Borel probability measures on $X$. Assume that $(\mu_n)_{n\geq 1}$ satisfies the weak LDP with rate function $I$. Then:
		\begin{enumerate}
			\item $(\mu_n)_{n\geq 1}$ satisfies the LDP with rate function $I$;
			\item $I$ is a proper function.
		\end{enumerate}
	\end{prop}
	
	The following statement establishes a criterion to determine whether the weak LDP holds, without knowing the rate function in advance. It will be the key tool to prove existence of the weak LDP in our context.
	
	\begin{prop}[{\cite[Thm.~4.1.11]{De-Ze}}]
		\label{criterion}
		Let $(\mu_n)_{n\geq 1}$ be a sequence of Borel probability measures on $X$. Define the function $I\colon X \to[0,\infty]$ by
		\begin{equation}
		\label{suplower}
		I(x)=\supl_{x\in V\emph{open}}-\lb(V) \text{ for all }x \in X.
		\end{equation}
		Then $I$ is lower semicontinuous. Moreover, if 
		\begin{equation}
		\label{supupper}
		I(x)=\supl_{x\in V\emph{open}}-\ub(V) \text{ for all }x \in X,
		\end{equation}
		then the sequence $(\mu_n)_{n\geq 1}$ satisfies the weak LDP with rate function $I$.
	\end{prop}
	
	Let us observe that, both in~\eqref{suplower} and in~\eqref{supupper}, we may clearly replace the whole collection of open sets containing the point $x\in X$ by any fundamental system of open neighborhoods of $x$. 
	
	\smallskip
	Assume now that $X$ is a locally convex, Hausdorff topological vector space over $\R$, and let $X^{*}$ denote its topological dual. In case the sequence $(\mu_n)_{n\geq 1}$ satisfies the LDP on $X$ with a proper, convex rate function $I$, it is possible to give an alternative expression for the rate function itself, provided that a certain logarithmic moment generating function exists. More precisely, define the \emph{logarithmic moment generating function} of the measure $\mu_n$, for each integer $n\geq 1$, as the function $\La_{n}\colon X^{*} \to (-\infty,\infty]$ given by
	\begin{equation*}
	\La_{n}(\varphi)=\log{\int_{X}e^{\langle \varphi,x\rangle }}d\mu_n(x) \quad\text{for all }\varphi\in X^{*},
	\end{equation*}
	where $\langle \cdot,\cdot \rangle$ denotes the standard dual pairing between $X^{*}$ and $X$. The \emph{limiting logarithmic moment generating function} of the sequence $(\mu_n)_{n\geq 1}$ is then defined as
	\begin{equation*} 
	\La(\varphi)=\limsupl_{n\to\infty}\frac{1}{n}\;\La_n(n\varphi) \in (-\infty,\infty]\quad\text{for all }\varphi \in X^{*}.
	\end{equation*}
	
	Given a function $f\colon X \to (-\infty,\infty]$, not identically infinite, we define its \emph{Fenchel-Legendre transform} $f^{*}\colon X^{*}\to (-\infty,\infty]$ as 
	\begin{equation*}
	f^{*}(\varphi)=\supl_{x\in X}\{\langle \varphi,x\rangle -f(x) \} \quad\text{for all }\varphi\in X^{*}. 
	\end{equation*}
	If $g\colon X^{*}\to (-\infty,\infty]$ is a function defined on the dual space, we shall view its Fenchel-Legendre transform $g^{*}$ as a function defined just on $X$, rather than on the entire bidual $X^{**}$.  
	
	A remarkable consequence of Varadhan's integral lemma (\cite[Thm.~4.3.1]{De-Ze}), in conjunction with Fenchel-Moreau's duality theorem (\cite[Thm.~1.11]{Brezis}), is the following characterization of the rate function (cf.~\cite[Thm.~4.5.10]{De-Ze}). 
	
	\begin{thm}
		\label{momentgen}
		Let $(\mu_n)_{n\geq 1}$ be a sequence of Borel probability measures on a locally convex, Hausdorff topological vector space $X$. Assume the following:
		\begin{enumerate}
			\item  the limiting logarithmic moment generating function $\La \colon X^{*}\to (-\infty,\infty]$ of the sequence $(\mu_n)_{n\geq 1}$ is finite for every $\varphi\in X^{*}$;		 
			\item  the sequence $(\mu_{n})_{n\geq 1}$ satisfies the LDP with a proper, convex rate function $I$. 
		\end{enumerate} 
		Then the rate function $I$ is the Fenchel-Legendre transform of $\La$, namely
		\begin{equation*}
		I(x)=\supl_{\varphi \in X^{*}}\{\langle \varphi,x\rangle -\La(\varphi)  \} \text{ for every }x\in X.
		\end{equation*}
	\end{thm}
	
	Theorem~\ref{momentgen} reveals the importance of knowing \emph{a priori} the existence of the LDP with a proper, convex rate function. 
	
	\section{Existence of LDP}
	\label{exist}

	We now set out to prove our main Theorem~\ref{main}. Specifically, the objective of the present section is twofold: in Proposition~\ref{existence}, we address existence of the weak LDP, with a certain rate function, under the pattern-avoiding assumption for the semigroup generated by the support of the driving measure, while in Proposition~\ref{fullLDP} the result is upgraded to the full LDP, under the additional requirement of finiteness of some exponential moment. Convexity of the rate function, and the ensuing identification of it as a Fenchel-Legendre transform, are dealt with in Section~\ref{conv}.
	
	\smallskip
	For a start, we briefly recall the setup. Let $G_1,\dots,G_r$ be a finite collection of non-trivial finitely generated groups of subexponential growth, $G=G_1\ast\cdots \ast G_r$ their free product. For any $i\in \{1,\dots,r\}$, $S_i\subset G_i$ is a finite set of generators of $G_i$, so that $S=\bigcup_{i=1}^{r}S_i$ is a finite generating set for $G$, with associated word length $\ell\colon G \to \N$. Let $\mu$ be a probability measure on $G$, $(Y_{n})_{n\geq 0}$ a right random walk on $G$ issued from the identity with steps distributed according to $\mu$. For every integer $n\geq 1$, let $\mu_n$ be the law of the random variable $\frac{1}{n}\ell(Y_n)$.
	
	\smallskip
	Henceforth, we shall denote by $B(y,\eps)$ the open interval $(y-\eps,y+\eps)\subset \R$, for any $y\in \R$ and any $\eps>0$. Furthermore, for any positive integer $k$, we let
	\begin{equation*}
	k B(y,\eps)=\{k z: z \in B(y,\eps)\}.
	\end{equation*}
	
	We precede the statement of Proposition~\ref{existence} by two technical lemmas which, taken together, essentially allow to reduce the problem of establishing LDP in this context to a setup akin to the standard case of i.i.d.~real random variables, in which (almost-)additivity of the process can be put to good use.
	
	The first of the two lemmas allows to deduce a lower bound for the asymptotic exponential decay rate of the probabilities $\mu_n(B(x,b))$ from a uniform lower bound on a non-lacunary sequence of times.	
	
	
	
	\begin{lem}
		\label{lowerbound}
		Suppose that there exist $a>0,\gamma \in \R$, a strictly increasing sequence $(n_k)_{k\geq 1}$ of positive integers with $\lim_{k\to\infty}n_{k+1}/n_k=1$, such that 
		\begin{equation}
		\label{multiplicative}
		\mu_{n_k}(B(x,a))\geq e^{n_k\gamma} \text{ for all }k\geq 1.
		\end{equation}
		Then, for all $b>a$, 
		\begin{equation*}
		\lb(B(x,b))\geq \gamma\;.
		\end{equation*}
	\end{lem}
	\begin{proof}
		Choose a finite set $\cF\subset G$ such that $\sum_{g\in \cF}\mu(g)>1/2$.
		For any $k\geq 1$, set 
		\begin{equation*}
		M_k=\sup\{\ell(x_1\cdots x_{n_{k+1}-n_k}):x_i \in \cF\cup\{e\} \text{ for all } 1\leq i\leq n_{k+1}-n_k \},
		\end{equation*}
		and notice that the upper bound $M_k\leq (n_{k+1}-n_k)M_1$ holds by subadditivity of $\ell$. 
		
		Now let $N\geq n_1$ be arbitrary; there exists a unique $k=k(N)\geq 1$ such that\linebreak $n_{k}\leq N<n_{k+1}$. 
		As $b-a>0$, the assumption $n_{k+1}/n_k\rightarrow 1$ implies that there exists $k_0\in \N$ such that 
		\begin{equation*}
		\{\ell(Y_{n_k})\in n_kB(x,a) \}\cap \{X_{n_k+1}\in \mathcal{F},\dots,X_N\in \mathcal{F} \}\subset \{\ell(Y_{N})\in NB(x,b) \} \text{ for all }k\geq k_0;
		\end{equation*}
		this follows from the double inequality $|\ell(g)-\ell(h)|\leq \ell(gh)\leq \ell(g)+\ell(h)$, holding for every $g,h \in G$.
		Now, if $k\geq k_0$ and $N\in\{n_{k},\dots,n_{k+1}-1\}$, we may estimate
		\begin{equation*}
		\begin{split}
		\mu_N(B(x,b))&=\P(\ell(Y_N)\in nB(x,b))\geq \P(\ell(Y_{n_k})\in n_kB(x,a),X_{n_k+1}\in \cF,\dots,X_{N}\in \cF)\\
		&=\mu_{n_k}(B(x,a))\mu(\cF)^{N-n_k}\geq e^{n_k\gamma}2^{-(n_{k+1}-n_k)}\;,
		\end{split}
		\end{equation*}
		the last two inequalities being given, respectively, by independence and stationarity of the process $(X_n)_{n\geq1}$, and by the assumption of the lemma. Taking the logarithm and dividing by $N$, we obtain
		\begin{equation*}
		\frac{1}{N}\log{\mu_N(B(x,b))}\geq \frac{n_k}{N}\gamma-\frac{n_{k+1}-n_k}{N}\log{2}\;.
		\end{equation*}
		Taking the inferior limit as $N\to\infty$ on both sides, and observing that the assumption on $(n_k)_k$ implies $\lim_{N\to\infty}n_{k(N)}/N=1$, we achieve the proof.
	\end{proof}
	
	The next lemma expresses the possibility of restricting the random walk to subsets on which the length function $\ell$ is almost additive, without losing consistently on the exponential decay rate of the probabilities involved.
	
	For every $T\in \R_{\geq 0}$, set $\theta_{T}=\sup\{|B^{G_i}(T)|:i=1,\dots,r \}$.
	\begin{lem}
		\label{weaklengthadd}
		Let $\nu$ be a probability measure on $G$, $\cT\subset G$ a finite subset avoiding patterns of type size $D$ for a certain integer $D>0$. Set $L\coloneqq \sup\{\ell(g):g\in \cT\}$. Then, for any $T\in \R_{\geq 0}$ and any set $F\subset B^{G}(T)\setminus\{e\}$,
		there exist a subset $A\subset F$ with $\nu(A)\geq (r\theta_{T})^{-2D}\nu(F)$ and an element $g\in \mathcal{T}$ such that at least one of the following holds:
		\begin{enumerate}
			\item for any integer $k\geq 1$ and any collection $g_1,\dots,g_k\in A$
			\begin{equation*}
			\ell(g_1\cdots g_k)\geq \ell(g_1)+\cdots +\ell(g_k)-k(2LD)\;;
			\end{equation*} 
			\item for any integer $k\geq 1$ and any collection $g_1g,\dots,g_kg\in A\cdot g$
			\begin{equation*}
			\ell(g_1g\cdots g_kg)\geq \ell(g_1)+\cdots +\ell(g_k)-k(2LD)\;.
			\end{equation*}
		\end{enumerate}
	\end{lem}
	Observe that $T/\log{\theta_T}\overset{T\to\infty}{\longrightarrow}\infty$ due to the subexponential growth of $G_1,\dots,G_r$; as a consequence, the factor $(r\theta_{T})^{-2D}$, quantifying the maximal loss in probability, is negligible on an exponential scale (cf.~the proof of Proposition~\ref{existence}).
	\begin{proof}
		The proof consists of a repeated application of the union bound for $\nu$, in order to extract various subsets of $F$ with predetermined letters in their reduced-word expression.
		
		To begin with, there exist $(i_1,j_1)\in \{1,\dots,r\}^{2}$ and $F_1\subset F$ such that $\nu(F_1)\geq r^{-2}\nu(F)$ and, for any $g\in F_1$, the first letter of $g$ is in $G_{i_1}$ and the last one is in $G_{j_1}$. If $i_1\neq j_1$, then $\ell(g_1\cdots g_k)=\ell(g_1)+\cdots \ell(g_k)$ for any $g_1,\dots,g_k\in F_1$, so that $A=F_1$ fulfils the statement. If $i_1=j_1$, we might choose a subset $E_1\subset F_1$ and elements $y_1,z_1\in G_{i_1}$ such that $\nu(E_1)\geq \theta_{T}^{-2}\nu(F_1)$ and, for each $g\in E_1$, the first letter of $g$ is $y_1$ and the last one is $z_1$. We distinguish three cases.
		\begin{itemize}
			\item[--] Suppose $\ell(y_1)>L,\ell(z_1)>L$; if $g$ is chosen in $\cT\setminus G_{i_1}$, it is easy to check that $\ell(g_1g\cdots g_kg)\geq \ell(g_1)+\cdots +\ell(g_k)-k(2L)$ for any $g_1g,\dots,g_kg\in E_1\cdot g$, so that we may set $A=E_1$ and conclude.
			\item[--] If just one between $y_1$ and $z_1$ has length exceeding $L$, or alternatively if	\linebreak $\ell(y_1)\leq L,\ell(z_1)\leq L$ and $z_1\neq y_1^{-1}$, then $\ell(g_1\cdots g_k)\geq \ell(g_1)+\cdots +\ell(g_k)-k(2L)$ for any $g_1,\dots,g_k\in E_1$; again, setting $A=E_1$ allows to conclude.  
		\end{itemize}
		The only remaining case is $\ell(y_1)\leq L,z_1=y_1^{-1}$. We then carry out the same procedure, selecting $F_2\subset E_1$, $(i_2,j_2)\in \{1,\dots,r\}^{2}$, with $\nu(F_2)\geq r^{-2}\mu(E_1)$ and so that, for each $g\in F_2$, the second letter of $g$ is in $G_{i_2}$ and the second-to-last one is in $G_{j_2}$. If $i_2\neq j_2$, then\linebreak $\ell(g_1\cdots g_k)\geq \ell(g_1)+\cdots +\ell(g_k)-k(2L)$ for any $g_1,\dots,g_k \in F_2$. If instead $i_2= j_2$, then choose $E_2\subset F_2$ and elements $y_2,z_2\in G_{i_2}$ so that $\nu(E_2)\geq \theta_T^{-2}\nu(F_2)$ and, for each $g\in E_2$, the second letter of $g$ is $y_2$ and the second-to-last one is $z_2$. Notice that, by assumption, $\cT$ is not contained in any conjugate of any factor $G_i$ by any word $\omega$ of type size not exceeding $D$. Therefore, unless $\ell(y_2)\leq L$ and $z_2=y_2^{-1}$, we can set $A=E_2$ and conclude as before. 
		
		Proceeding in this way, we select, if needed at each successive step, nested subsets\linebreak $E_2\supset F_3\supset E_3\supset\cdots \supset E_{D}$. The set $E_D$ has the property that $\nu(E_D)\geq (\theta_{T})^{-2}\nu(F_{D})\geq (r\theta_T)^{-2D}\nu(F)$; furthermore, there are letters $y_3,\dots,y_{D},z_{D}$ such that, for any $g\in E_{D}$, the reduced-word expression of $g$ is $y_1\cdots y_{D}\cdots z_{D}y_{D-1}^{-1}\cdots y_{1}^{-1}$. It remains to deal with three possibilities, as above.
		\begin{itemize}
			\item[--] Suppose $\ell(y_D)>L,\ell(z_D)>L$, and set $\omega=y_{1}\cdots y_{D-1}$. If $g$ is chosen in\linebreak $\mathcal{T}\setminus \omega G_{i_{D}}\omega^{-1}$, where $G_{i_D}$ is the factor to which both $y_D$ and $z_D$ belong\footnote{To select an element $g$ of this sort, concatenate any letter $y'_{D}$ with $\omega$, in such a way that $\omega y'_{D}$ is a reduced word; using that $\cT$ avoids patterns of type size $D$, pick $g\in \mathcal{T}$ not starting with $\omega y'_{D}$ nor ending with $(\omega y'_{D})^{-1}$.}, then\linebreak $\ell(g_1g\cdots g_kg)\geq \ell(g_1)+\cdots +\ell(g_k)-k(2DL)$ for any $g_1g,\dots,g_kg\in E_D\cdot g$, so that we may set $A=E_D$ and conclude.
			\item[--] If just one between $y_D$ and $z_D$ has length exceeding $L$, or alternatively if\linebreak $\ell(y_D)\leq L,\ell(z_D)\leq L$ and $z_D\neq y_D^{-1}$, then $\ell(g_1\cdots g_k)\geq \ell(g_1)+\cdots +\ell(g_k)-k(2DL)$ for any $g_1,\dots,g_k\in E_1$; conclude by setting $A=E_1$.
			\item[--] Finally, assume $\ell(y_1)\leq L,z_D=y_{D}^{-1}$, and choose $g\in \mathcal{T}$ not starting with $y_1\cdots y_D$ nor ending with $(y_1\cdots y_D)^{-1}$. Then again $\ell(g_1g\cdots g_kg)\geq \ell(g_1)+\cdots +\ell(g_k)-k(2DL)$ for any $g_1g,\dots,g_kg\in E_D\cdot g$. The set $A=E_{D}$ satisfies the conclusion.
		\end{itemize}
		The argument is finalized.
	\end{proof} 
	If a set $A$ (resp.~$A\cdot g$) satisfies the conclusion of Lemma~\ref{weaklengthadd}, then we say that $A$ (resp.~$A\cdot g$) has the \emph{weak length additivity property} of order $LD$.
	
	\smallskip
	We are now in a position to prove existence of the weak LDP.
	\begin{prop}
		\label{existence}
		Let $G,S,\ell,\mu$ be as above, $(Y_n)_{n\geq 0}$ a right random walk on $G$ issued from the identity with increments distributed according to $\mu$. Suppose that the support of $\mu$ generates a pattern-avoiding semigroup $\Ga\subset G$. Then the sequence of $\R$-valued random variables $\bigl(\frac{1}{n}\ell(Y_n)\bigr)_{n\geq 1}$ satisfies the weak LDP with a rate function $I\colon \R_{\geq 0}\to [0,\infty]$.
	\end{prop}
	
	
	

	\begin{proof}	
		We rely on the criterion phrased in Proposition~\ref{criterion}, checking that the condition expressed therein is satisfied. Arguing by contradiction, suppose that there exists $x\in \R_{\geq 0}$ such that 
		\begin{equation}
		\label{firstcontr}
		I(x)\neq\supl_{x\in V\text{open}}-\ub(V).
		\end{equation}
		As the left-hand side of~\eqref{firstcontr} always dominates the right-hand side by definition, this yields
		\begin{equation}
		\label{secondcontra}
		I(x)>\supl_{x\in V\text{open}}-\ub(V).
		\end{equation}
		Notice first that, necessarily, $x$ is strictly positive; indeed, for $x=0$ the criterion in Proposition~\ref{criterion} is trivially satisfied, as $\lim_{n}\frac{1}{n}\log{\mu_n}(B(0,\eps))$ exists in $[-\infty,0]$ for every $\eps>0$, by subadditivity of $\ell$.

		As a consequence of~\eqref{secondcontra}, there exist $\delta,\eta>0$ such that 
		\begin{equation}
		\label{contr}
		-\lb(B(x,\delta))>\biggl(\supl_{\rho>0}-\ub(B(x,\rho))\biggr) +\eta\;.
		\end{equation}
		Fix a positive real number $\rho$ such that $\rho<\inf\{x,\delta\}$; then, by~\eqref{contr}, there are infinitely many positive integers $n_j,j\geq 1$, for which
		\begin{equation}
		\label{secondcontr}
		\lb(B(x,\delta))<\frac{1}{n_j}\log{\mu_{n_j}((B(x,\rho)))}\;\; -\eta\;.
		\end{equation}
		For notational simplicity, denote by
		\begin{equation}
		\label{notation}
		\alpha=\lb(B(x,\delta)),\; \beta_j=\frac{1}{n_j}\log{\mu_{n_j}((B(x,\rho)))} \text{ for every }j\geq 1.
		\end{equation} 
		We claim that, if $j$ is taken to be sufficiently large, the inequality $\alpha\geq \beta_j-\eta$ holds, which is opposite to what is given by~\eqref{secondcontr}, giving the desired contradiction. 

		The hypothesis on the semigroup $\Ga$ ensures the existence of a finite subset $\mathcal{T}\subset \Ga\setminus\{e\}$ with the following property: there exists an integer $D>0$ such that, for any reduced word $\omega$ of type size $D$ in $G$, we can find $g\in \mathcal{T}$ not starting in $\omega$ and not ending in $\omega^{-1}$ (cf.~Section~\ref{patternavoid}). For any $g\in \mathcal{T}$, choose $t(g)\in \N_{\geq 1}$ and $p(g)\in \R_{>0}$ such that the random walk attains $g$ in $t(g)$ steps with probability $p(g)$, that is $\P(Y_{t(g)}=g)=p(g)$. 
		Define $L=\sup\{\ell(g):g\in \mathcal{T}\}$, $p=\inf\{p(g):g\in \mathcal{T}\}$, $t=\sup\{t(g):g\in \mathcal{T} \}$. Keeping with our earlier notation, let
		\begin{equation*} \theta_{T}=\sup\{|B^{G_i}(T)|:i=1,\dots,r \} \quad \text{for any } T\in \R_{\geq 0}.
		\end{equation*}
		Now choose an integer $j_0\geq 1$ so that 
		\begin{equation*}
		n_{j_0}\geq \sup\bigl\{ 
		(2LD+tx)(\delta-\rho)^{-1},\eta^{-1}\bigl(2D(\log{r}+\log{\theta_{n_{j_0}(x+\rho)}})-\log{p}\bigr) \bigr\}\;;
		\end{equation*}
		this exists since $T/\log{\theta_{T}}\overset{T\to\infty}{\longrightarrow}\infty$ by the subexponential-growth assumption on the factors $G_1,\dots,G_r$. 
		Define $F=\{g \in G: \ell(g)\in n_{j_0}B(x,\rho) \}$, so that  $e^{\beta_{j_0}n_{j_0}}=\P(Y_{n_{j_0}}\in F)$ by~\eqref{notation}.
		Notice also that $F$ does not contain the identity as $n_{j_0}(x-\rho)>0$.
		Applying Lemma~\ref{weaklengthadd}, with $\nu$ being the law of the random variable $Y_{n_{j_0}}$, we can manufacture a set $A\subset F$ and an element $g\in \mathcal{T}$ such that
		\begin{itemize}
			\item[--] $\P(Y_{n_{j_0}}\in A)\geq (r\theta_{n_{j_0}(x+\rho)})^{-2D}e^{\beta_{j_0}n_{j_0}}$ and
			\item[--] either $A$ or $A\cdot g$ has the weak length additivity property of order $LD$.
		\end{itemize}
		We distinguish two cases.
		\begin{itemize}
			\item First case: $A$ has the weak length additivity property of order $LD$.
			
			Define the sequence $n_k=kn_{j_0},k\geq 1$. Since $n_{j_0}\geq 2LD(\delta-\rho)^{-1}$, there exists $\rho'<\delta$ such that $\rho'-\rho\geq 2n_{j_0}^{-1}LD$. For such a choice of $\rho'$, we have that\linebreak $\ell(g_1\cdots g_k)\in n_{k}B(x,\rho')$ whenever $g_1,\dots,g_k$ are chosen from $A$. 
			Therefore, we may estimate, for each $k\geq 1$,
			\begin{equation*}
			\begin{split}
			\quad\quad\;\mu_{n_k}(B(x,\rho'))&=\P(\ell(Y_{n_k})\in n_kB(x,\rho'))\geq \P(X_1\cdots X_{n_{j_0}}\in A,\dots,X_{n_{k-1}+1}\cdots X_{n_k}\in A)\\
			&\geq \P(Y_{n_{j_0}}\in A)^{k}\geq \bigr((r\theta_{n_{j_0}(x+\rho)})^{-2D}e^{\beta_{j_0}n_{j_0}}\bigr)^{k}\geq  e^{n_k(\beta_{j_0}-\eta)}\;,
			\end{split}
			\end{equation*}
			where the middle inequality is given by independence and stationarity\linebreak of the process $(X_{n})_{n\geq 1}$, while the last one comes from our choice\linebreak $n_1=n_{j_0}\geq 2D\eta^{-1}(\log{r}+\log{\theta_{n_{j_0}(x+\rho)}})$. Lemma~\ref{lowerbound} gives
			\begin{equation*} \alpha=\liminf_{n\to\infty}\frac{1}{n}\log{\mu_nB(x,\delta)}\geq \beta_{j_0}-\eta\;,
			\end{equation*}
			as desired.
			\item Second case: $A\cdot g$ has the weak length additivity property of order $LD$.
			
			Define the sequence $n_k=k(n_{j_0}+t(g)),k\geq 1$. Since $n_{j_0}\geq (\delta-\rho)^{-1}(2LD+tx)$,\linebreak it is possible to select $\rho'<\delta$ so that $\rho'-\rho\geq n_{j_0}^{-1}(2LD+tx)$. It is straightforward to verify that this choice of $\rho'$ ensures $\ell(g_1g\cdots g_kg)\in n_kB(x,\rho')$ for every $g_1,\dots,g_k\in A$. 
			
			As before, we may thus estimate
			\begin{equation*}
			\begin{split}
			\quad\quad\;\mu_{n_k}(B(x,\rho'))&=\P(\ell(Y_{n_k})\in n_kB(x,\rho'))\geq \P(X_1\cdots X_{n_{j_0}}\in A,X_{n_{j_0}+1}\cdots X_{n_{j_0}+t(g)}=g)^{k}\\
			&\geq \P(Y_{n_{j_0}}\in A)^{k}p(g)^{k}\geq \bigr((r\theta_{n_{j_0}(x+\rho)})^{-2D}e^{\beta_{j_0}n_{j_0}}\bigr)^{k}p^{k}\geq  e^{n_k(\beta_{j_0}-\eta)} 
			\end{split}
			\end{equation*}
			for each $k\geq 1$. This time, the last inequality stems from our choice\linebreak $n_{j_0}\geq \eta^{-1}\bigl(2D(\log{r}+\log{\theta_{n_{j_0}(x+\rho)}})-\log{p}\bigr)$. Applying Lemma~\ref{lowerbound} once more, we deduce again that $\alpha\geq \beta_{j_0}-\eta$.
		\end{itemize}
		The proof is concluded.
	\end{proof}
	
	\begin{prop}
		\label{fullLDP}
		In the setting of Proposition~\ref{existence}, assume further that $\mu$ has a finite exponential moment. Then the rate function $I$ governing the weak LDP for the sequence $\bigl(\frac{1}{n}\ell(Y_n)\bigr)_{n\geq 1}$ is proper, and the sequence $\bigl(\frac{1}{n}\ell(Y_n)\bigr)_{n\geq 1}$ satisfies the full LDP with rate function $I$.
	\end{prop}
	\begin{proof}
		As before, we let $\mu_n$ be the law of the random variable $\frac{1}{n}\ell(Y_n)$, for every $n\geq 1$. In light of Proposition~\ref{weakstrong}, it suffices to show that the sequence $(\mu_n)_{n\geq 1}$ is exponentially tight. 
		By the assumption, there exists a real number $\tau>0$ such that $C\coloneqq\int_{G}\exp{(\tau\ell(g))}d\mu(g)<\infty$.
		
		Fix $M>0$. Then
		\begin{equation*}
		\mu_n([0,M]^{\mathsf{c}})=\P(\ell(Y_n)> nM)=\P(\exp{\tau\ell(Y_n)}>\exp{\tau nM})\leq \frac{\E[\exp{\tau\ell(Y_n)}]}{\exp{\tau nM}}\;,
		\end{equation*}
		the last upper bound being given by Markov's inequality.  Subadditivity of the length function $\ell$, together with independence and stationarity of the process $(X_n)_{n\geq1 }$, gives
		\begin{equation*}
		\begin{split}
		\E[\exp(\tau\ell(Y_n))]&\leq\E\biggl[\exp{\tau\biggl(\sum_{i=1}^{n}\ell(X_n)\biggr)}\biggr]=\E\biggl[\prod_{i=1}^{n}\exp{\tau\ell(X_i)}\biggr]=\prod_{i=1}^{n}\E[\exp{\tau\ell(X_i)}]\\
		&=(\E[\exp{\tau\ell(X_1)}])^{n}=\biggl(\int_{G}e^{\tau \ell(g)}d\mu(g)\biggr)^{n}=C^{n} \;\text{ for every }n\geq 1.
		\end{split}
		\end{equation*}
		Combining the previous two estimates, taking the logarithm and dividing by $n$, we obtain $\frac{1}{n}\log{\mu_n([0,M]^{\mathsf{c}})}\leq \log{C}-\tau M$.
		As a result,
		\begin{equation*}
		\limsupl_{n\to\infty}\frac{1}{n}\log{\mu_n([0,M]^{\mathsf{c}})}\overset{M\to\infty}{\longrightarrow}-\infty\;,
		\end{equation*}
		which establishes exponential tightness of the sequence $(\mu_n)_{n\geq 1}$.
	\end{proof}
	
	\section{Convexity of the rate function}
	\label{conv}
	
	The chief aim of this section is the proof of convexity of the rate function associated to the LDP for the sequence $\bigl(\frac{1}{n}\ell(Y_n)\bigr)_{n\geq 1}$. In the last part, we gather some further properties of the rate function, and deduce its characterization expressed in the last sentence of Theorem~\ref{main}. As in the foregoing section, we let $\mu_n$ denote the law of the random variable $\frac{1}{n}\ell(Y_n)$, for $n\geq 1$.
	
	Recall that, if $X$ is a real vector space, a function $f\colon X \to (-\infty,+\infty]$ is convex if, for any $x_1,x_2 \in X$ and any $\lambda\in  [0,1]$,
	\begin{equation}
	\label{convexity}
	f(\lambda x_1+(1-\lambda)x_2)\leq \lambda f(x_1)+(1-\lambda)f(x_2)\;;
	\end{equation} 
	the function $f$ is mid-point convex if the previous inequality holds for $\lambda=1/2$, that is if
	\begin{equation*}
	f\biggl(\frac{1}{2} x_1+\frac{1}{2}x_2\biggr)\leq \frac{1}{2} f(x_1)+\frac{1}{2}f(x_2)
	\end{equation*} 
	for all $x_1,x_2\in X$. 
	
	Suppose now $X$ is a topological (real) vector space. By iteration, a mid-point convex function $f$ satisfies the inequality~\eqref{convexity} for any $\lambda \in \{k/2^n:n\in \N,k\in \{0,\dots, 2^n \}\}$. The latter set being dense in $[0,1]$, ~\eqref{convexity} can be extended to all $\lambda\in [0,1]$ by a standard approximation argument, provided that we know that $f$ is lower semicontinuous. To wrap up, a lower semicontinuous, mid-point convex function $f\colon X\to (-\infty.+\infty]$ is convex.
	
	\begin{prop}
		\label{convex}
		Let $G,S,\ell,\mu, (Y_n)_{n\geq 0}$ be as in Proposition~\ref{existence}. Then the rate function $I$, governing the LDP for the sequence of $\R$-valued random variables $\bigl(\frac{1}{n}\ell(Y_n)\bigr)_{n\geq 1}$, is convex.
	\end{prop}
	
	The proof bears a lot of resemblance with the proof of Proposition~\ref{existence}; for the sake of conciseness, we shall omit a few details.	
	\begin{proof}
		As observed in the previous paragraph, it suffices to show that $I$ is mid-point convex, since we already know $I$ that is lower semicontinuous. Again, we argue by contradiction: assume there exist $x_1<x_2\in \R$ such that
		\begin{equation}
		\label{nomidpoint}
		I\biggl(\frac{1}{2} x_1+\frac{1}{2}x_2\biggr)> \frac{1}{2} I(x_1)+\frac{1}{2}I(x_2)\;.
		\end{equation} 
		Recall that we have
		\begin{equation*}
		I(x)=\supl_{x\in V\text{open}}-\lb(V)=\supl_{x\in V\text{open}}-\ub(V) \quad \text{for all }x \in X;
		\end{equation*}
		therefore,~\eqref{nomidpoint} implies that there exist $\delta,\eta >0$ such that
		\begin{equation}
		\label{wrongineq}
		\begin{split}
		\ub\biggl(&B\biggl(\frac{1}{2}x_1+\frac{1}{2}x_2,\delta\biggr)\biggr)
		<\\
		&<\frac{1}{2}\biggl(\lb(B(x_1,\rho_1))+\lb(B(x_2,\rho_2))\biggr)-\eta
		\end{split}
		\end{equation}
		for any $\rho_1,\rho_2>0$. 
		Notice that this forces in particular $x_1,x_2\in \R_{\geq 0}$. Choose $\rho\coloneqq \rho_1=\rho_2<\delta$. For a sufficiently large $n_0$ and every $n\geq n_0$, we claim that there exists $\phi(n)\in\{2n,\dots,2n+t \}$ such that 
		\begin{equation}
		\label{rightineq}
		\frac{1}{\phi(n)}\log{\mu_{\phi(n)}\biggl(B\biggl(\frac{1}{2}x_1+\frac{1}{2}x_2,\delta\biggr)\biggr)}\geq \frac{1}{2}\biggl(\frac{1}{n}\log{\mu_n(B(x_1,\rho))}+\frac{1}{n}\log{\mu_n(B(x_2,\rho))}\biggr)-\eta\;.
		\end{equation}
		Letting $n$ vary over an arithmetic progression for which the corresponding sequence of $\phi(n)$ is strictly increasing, it is clear that we obtain a contradiction to~\eqref{wrongineq}.
		
		It remains to prove the claim just stated. Let $\mathcal{T}\subset \Ga\setminus\{e\}$ be a finite set avoiding patterns of type size $D$, and fix $n\geq n_0$; let $F_i=\{g\in G:\ell(g)\in nB(x_i,\rho) \}$, $i=1,2$. Adapting the proof of Lemma~\ref{weaklengthadd} appropriately\footnote{There is a minor nuisance here if $x_1=0$, as $F_1$ contains the identity; replacing $F_1$ with $F_1\setminus\{e\}$ results in harmless modifications of the probabilities involved.}, we deduce that there is an element $g\in \mathcal{T}$ and subsets $A_i\subset F_i$ such that $\P(Y_n\in A_i)\geq (r\theta_{n(x+\rho_i)})^{-D}\P(Y_{n}\in F_i)$ and
		\begin{itemize}
			\item[--] either for any $g_1\in A_1,g_2\in A_2$ it holds $\ell(g_1g_2)\geq \ell(g_1)+\ell(g_2)-2LD$,
			\item[--] or for any $g_1\in A_1,g_2\in A_2$, $\ell(g_1gg_2)\geq\ell(g_1)+\ell(g_2)-2LD$.
		\end{itemize}
		In the first case, we get the inequality~\eqref{rightineq} for $\phi(n)=2n$, by observing that $g_1\in A_1,g_2\in A_2$ imply $\ell(g_1g_2)\in 2nB((x_1+x_2)/2,\delta)$; in the second case, we get it for $\phi(n)=2n+t(g)$. We refer to the proof of Proposition~\ref{existence} for the missing details.
	\end{proof}
	
	\subsection{Further properties of the rate function}
	\label{further}
	
	We list below some additional properties of the rate function, emphasizing connections with other relevant quantities associated to the random walk, such as the rate of escape and the spectral radius.
	
	\begin{enumerate}
		\item Since $\frac{1}{n}\ell(Y_n)$ converges to the escape rate $\lambda$ almost surely, $I$ has a zero at $x=\lambda$. 
		\item Convexity of the rate function $I$ gives, as an immediate corollary, that its effective domain $D_I$ is a convex subset of $\R_{\geq 0}$, hence a (possibly degenerate\footnote{In general, the rate function $I$ can be as degenerate as possible: for instance, if $G=\langle a,b \rangle$ is a free group on two generators, and $\mu(a)=p=1-\mu(b)$ for some $p \in [0,1]$, then $I(1)=0$ and $I(x)=\infty$ for any $x\in \R_{\geq 0}\setminus\{1\}$, as $\ell(Y_n)=n\;$ $\P$-almost surely for every $n$.}) sub-interval of the positive half-line. 
		Standard properties of convex functions defined on sub-intervals of the real line imply that, on the open interval $D_{I}^{\circ}$, the rate function $I$ is continuous, admits left and right derivatives at every point, and it is differentiable outside a countable set of points. In particular, continuity on $D_I^{\circ}$ gives that 
		\begin{equation*}
		\liml_{n\to\infty}\frac{1}{n}\log{\mu_n}(V)=-\inf_{x\in V}I(x) \;\text{ for every open set } V\subset D_{I}^{\circ}\;;
		\end{equation*}
		in other words, the exponential decay rate of the sequence $(\mu_{n}(V))_{n\geq 1}$ is well-defined whenever $V$ is an open subset of $D_I^{\circ}$.
		\item Define the \emph{spectral radius} of the random walk as 
		\begin{equation*}
		\rho=\limsup\limits_{n\to\infty}\P(Y_n=e)^{\frac{1}{n}} \in [0,1]\;.
		\end{equation*} 
		If the measure $\mu$ is symmetric, that is $\mu(g)=\mu(g^{-1})$ for every $g \in G$, this quantity coincides with the spectral radius of the \emph{Markov operator} associated with the random walk (cf.~\cite[Chap.~6]{Lyons-Peres}). 
		For every $\delta>0$, we have
		\begin{equation*}
		\mu_{n}(B(0,\delta))=\mu_n([0,\delta))\geq\mu_{n}(0)=\P(\ell(Y_n)=0)=\P(Y_n=e),
		\end{equation*}
		which implies
		\begin{equation*}
		\ub(B(0,\delta))\geq \limsup\limits_{n\to\infty}\frac{1}{n}\log{\P(Y_n=e)}=\log{\rho}\;,
		\end{equation*}
		with the understanding that $\log{\rho}=-\infty$ if $\rho=0$.
		The previous inequality holding for every $\delta>0$, we infer that
		\begin{equation}
		\label{ineq}
		I(0)=\supl_{\delta >0}-\ub(B(0,\delta))\leq -\log{\rho}\;.
		\end{equation}
		As a consequence, we deduce that $0\in D_I$ provided that the spectral radius is strictly positive. This occurs, for instance, whenever the semigroup $\Ga$ generated by $\supp{\mu}$ contains $e$: if $n_0\in \N$ is any integer for which $\P(Y_{n_0}=e)>0$, then
		\begin{equation*} \rho\geq\limsup_{k\to\infty}\P(Y_{kn_0}=e)^{\frac{1}{kn_0}}\geq \limsup_{k\to\infty}\bigl(\P(Y_{n_0}=e)^{k}\bigr)^{\frac{1}{kn_0}}>0. 
		\end{equation*}
		It is worth mentioning that equality $I(0)=-\log{\rho}$ actually holds\footnote{We thank S.~M\"{u}ller for communicating this fact.}, whenever the LDP for the word length functional is verified and the measure $\mu$ driving the random walk satisfies $\inf\{\mu(g):g\in \supp{\mu} \}>0$  (see~\cite[Lem.~2.8]{Mueller}). 
		\item As far as the least upper bound of $D_I$ is concerned, assume that the support of $\mu$ is bounded, and let $L=\sup\{\ell(g):g\in \supp{\mu}\}<\infty$. Then $I\equiv \infty$ on the open half-line $(L,\infty)$, as subadditivity of $\ell$ implies $\ell(Y_n)\leq nL$ $\P$-almost surely for any $n\geq 1$. Therefore, in this case, $D_I$ is contained in $[0,L]$.
		
		If no restriction is placed on the size of $\supp{\mu}$, then $\sup{D_I}$ may be infinite\footnote{Consider, once again, $G=\langle a,b\rangle$ a free group on two generators, and choose a measure $\mu$ with $\supp{\mu}=\langle a \rangle$. Then $\P(\ell(Y_n)=nk)\geq (\mu(a^{k}))^{n}$ for all integers $n,k\geq 1$, so that $I(k)<\infty$ for any $k\geq 1$. In this example, we have thus $D_I=\R_{\geq 0}$.}. 
	\end{enumerate}

	\subsection{The rate function as a Fenchel-Legendre transform}
	\label{FL}
	It remains to prove the final statement of Theorem~\ref{main}, under the assumption that $\mu$ has finite moment-generating function. By virtue of Theorem~\ref{momentgen}, it suffices to prove that the limiting logarithmic moment generating function of the sequence $(\mu_n)_{n\geq 1}$, given by
	\begin{equation*}
	\Lambda(z)=\limsupl_{n\to\infty}\frac{1}{n}\log{\int_{\R}e^{nz\cdot x}d\mu_{n}(x)}=\limsupl_{n\to\infty}\frac{1}{n}\log{\E[e^{ z\cdot\ell(Y_n)}]}  \;,\; z \in \R,
	\end{equation*}
	is finite everywhere, where we have canonically identified $\R$ with its dual space, and the dual pairing with the standard product of real numbers.  
	
	Fix $z\in \R_{\geq 0}\;$; then $\E[e^{z\cdot \ell(Y_1)}]=\int_{G}\exp{(z\ell(g))}\;d\mu(g)<\infty$, since all exponential moments of $\mu$ are finite. Moreover, for any $n,m\geq 1$, we have
	\begin{equation*}
	\E[e^{z\cdot \ell(Y_{n+m})}]\leq \E[e^{z\cdot \ell(X_1\cdots X_n)}e^{z\cdot \ell(X_{n+1}\cdots X_{n+m})}]=\E[e^{z\cdot \ell(Y_n)}]\E[e^{z\cdot \ell(Y_m)}]\;;
	\end{equation*}
	the first inequality comes from subadditivity  of the length function $\ell$, whereas the second follows from independence and stationarity of the process $(X_n)_{n\geq 1}$. Therefore, the sequence
	\begin{equation}
	\label{Fekete}
	a_n=\log{\E[e^{z\cdot \ell(Y_n)}]}\;,\;n\geq 1,
	\end{equation} 
	is subadditive, that is $a_{n+m}\leq a_n+a_m$ for every $n,m\geq 1$; Fekete's lemma (\cite[Ex.~3.9]{Lyons-Peres}) gives
	\begin{equation*}
	\Lambda(z)=\liml_{n\to\infty}\frac{1}{n}\log{\E[e^{ z\cdot\ell(Y_n)\rangle}]}=\infl_{n\geq 1}\frac{1}{n}\log{\E[e^{ z\cdot\ell(Y_n)\rangle}]}\leq \E[e^{z\cdot \ell(Y_1)}]<\infty\;.
	\end{equation*}
	If $z\in \R_{<0}\;$, a similar argument shows that the sequence~\eqref{Fekete} is superadditive, and $\La(z)<\infty$ follows all the same.

	\section{Concluding remarks and open questions}
	\label{final}
	
	\subsection{Groups with strongly connected finite-state automata}
	We mention another class of examples to which our method would apply: finitely generated groups whose cone type automaton with respect to a given generating set is finite and strongly connected.
	
	Let $G$ be a finitely generated group, $S$ a finite set of generators, $\ell$ the word length defined by $S$ on $G$. For every element $g\in G$, we define the \emph{cone type} of $g$ as the set 
	\begin{equation*}
	C(g)=\{h \in G:\ell(gh)=\ell(g)+\ell(h)  \}.
	\end{equation*} 
	Notice that the usual definition of cone type which appears in the literature (\cite{Bridson-Haefliger,Epstein,Ohshika}) involves geodesic words in the alphabet $S$, rather that actual group elements of $G$; our definition is more convenient for the purposes of this discussion.

	The cone type of an element selects those geodesic segments that can be attached (in algebraic terms, multiplied) to it on the right so that the concatenation is again a geodesic segment. Observe that it is precisely this notion that, implicitly, comes into play both in the proof of existence of LDP and in the proof of convexity of the rate function. 
	
	Cone types offer an algorithmic way to label geodesics in the group $G$, in other words to identify those strings $(s_1,\dots,s_n)$ of letters in the alphabet $S$ such that $\ell(s_1\cdots s_n)=n$. This is achieved through the construction of a finite state automaton (cf.~\cite{Epstein}), called the \emph{cone type automaton} of $G$ with respect to the language given by $S$. Assume there are only finitely many cone types $C_0=C(e),C_1,\dots,C_s$, which we view as vertices of a directed graph $\Delta$ whose edges are labelled by elements of $S$; more precisely, we connect the cone type $C(g)$ of an element $g$ to the cone type of $C(gs)$, via a directed edge labelled by $s \in S$, if and only if $s \in C(g)$. It is immediate that the definition doesn't depend on the choice of $g$ but only on its cone type. If $e\notin S$, there is a one-to-one correspondence between edge-paths in the directed graph $\Delta$ starting at $C_0$ and finite sequences $(s_1,\dots,s_n)\in S^n$ such that $\ell(s_1\cdots s_n)=n$, that is geodesic words in the alphabet $S$. 
	
	Now, the conditions we need to impose in order for the arguments of Sections~\ref{exist} and~\ref{conv} to carry over unaffectedly are:
	\begin{enumerate}
		\item the finite directed graph $\Delta$ is strongly connected, meaning that there is a directed path joining any two of its vertices;
		\item every element of $G$ belongs to the cone type of some non-trivial element; otherwise stated, for any geodesic word $\omega=(s_1,\dots,s_n)$ in the alphabet $S$, there is a cone type $C\neq C_0$ from which we can follow a directed path in the graph $\Delta$ according to the labelling given by $\omega$. 
	\end{enumerate}
	
	
	\begin{ex}[Simple random walks on integer lattices]
		Consider $G=\Z^d$ with its standard symmetric set of generators $S=\{\pm e_i:1\leq i\leq d \}$. Any probability distribution $\mu$ with $\supp{\mu}\subset S$ gives rise to a simple random walk $(Y_n)_{n\in \N}$ on $\Z^d$. It is clear that there are exactly $2^d+2d+1$ different cone types (the $2^d$ quadrants, the $2d$ half-spaces delimited by the $d$ coordinate planes, and the whole $\Z^d$). It takes a moment to realize that both conditions stated above are met. We thus recover, by elementary means, existence of the LDP with convex rate function for the process $\frac{1}{n}\norm{Y_n}_1$ (where $\norm{(x_1,\dots,x_d)}_1=|x_1|+\cdots |x_d|$ for any $(x_1,\dots,x_d)\in \R^d$), which is usually seen as a straightforward consequence of Cramer's theorem for the empirical mean of i.i.d.~random vectors (see~\cite[Thm.~2.2.30]{De-Ze}).
	\end{ex}
	
	Finiteness of the number of cone types appears to be an intrinsic requirement when attempting to establish the LDP via the strategy presented here, while the two additional conditions on the cone type automaton mentioned above can be presumably lifted through a refinement of the method. 
	
	A large class of finitely generated groups having only finitely many cone types, with respect to any finite generating set, is given by Gromov-hyperbolic groups; indeed, in such groups the cone type of an element only depends on its $k$-\emph{tail}, for a fixed positive integer $k$ depending only on the group (see~\cite{Bridson-Haefliger}). Our considerations thus provide substance to the claim that Theorem~\ref{main} holds for any Gromov-hyperbolic group\footnote{(Added in revision) Gou\"{e}zel has shown (\cite[Lem.~2.4]{Gouezel}) that a non-elementary hyperbolic group $G$ equipped with a word length $\ell$ satisfies the following geometric property: there exist constants $c,C>0$ such that, for any $x,y\in G$, there is an element $a\in G$ of length at most $C$ such that $\ell(xay)\geq \ell(x)+\ell(y)-c$. The result has been subsequently extended in~\cite[Lem.~5.3]{Dussaule} to relatively hyperbolic groups. It can be used as a replacement of almost length additivity throughout the proof of Theorem~\ref{main}, thereby proving its validity for irreducible random walks on any relatively hyperbolic group, with respect to any word length. The resulting argument simplifies the proof of~\cite[Thm.~1.2]{Boulanger-Mathieu-Sert-Sisto}, which however addresses more general spaces and walks, and yields a finer result on the rate function.}.

	
	\subsection{Some open problems}
	
	Computing the exact expression of the rate function, in the cases treated by Theorem~\ref{main}, is mostly out of reach; however, it is worth carrying through the computation in the easiest case of symmetric simple random walks on free groups, to get a flavour of what should happen in more general circumstances. This has already been performed in~\cite{Cagri-thesis}: let $G$ be a free group on $r\geq 1$ generators, $S=\{a_1,\dots,a_r\}$ a free generating set, $\mu$ the uniform probability measure on $S\cup S^{-1}$, i.e.~$\mu(a_{i})=\mu(a_i^{-1})=(2r)^{-1}$ for any $i \in \{1,\dots,r\}$. The rate function governing the LDP for the sequence $\bigl(\frac{1}{n}\ell(Y_n)\bigr)_{n\geq 1}$ is given by the following expression:
	\begin{equation*}
	I(x)=
	\begin{cases}
	\frac{1+x}{2}\log{(1+x)}+\frac{1-x}{2}\log{(1-x)}+\log{r}-\frac{1+x}{2}\log{(2r-1)} &\text{ if }x\in [0,1],\\
	\infty & \text{ otherwise },
	\end{cases}
	\end{equation*}
	where we agree that $0\log{0}=0$. The function $I$ is analytic in $(0,1)$ and strictly convex in its effective domain $[0,1]$, and hence admits a unique zero at $\lambda=1-1/r$, corresponding to the escape rate of the random walk; as a consequence thereof, the probability $\P\bigl(|\frac{1}{n}\ell(Y_n)-\lambda|\geq \eps\bigr)$ that the renormalized length deviates largely from the escape rate decays exponentially fast with $n$ for any $\eps>0$. Furthermore, the value of $I$ at $0$ is equal (in absolute value) to the logarithm of the spectral radius, as expected. Lastly, we notice that the right derivative $I'(0)$ at $0$ is finite, while the left derivative $I'(1)$ at $1$ is infinite.  
	
	\smallskip
	This motivates the following questions:
	
	\begin{enumerate}
		\item Is the rate function $I$ in Theorem~\ref{main} always strictly convex? In particular, does it always have a unique zero at $x=\lambda$?
		\item What are the finer regularity properties of the rate function? What is the behaviour of the (one-sided) derivatives of $I$ at the extreme points of its effective domain? 
	\end{enumerate}
	
	Assuming the validity of Theorem~\ref{main} for Gromov-hyperbolic groups, the same questions can obviously be phrased in this broader context as well.





	


\footnotesize
\begin{thebibliography}{99}
	
	\bibitem{Abels-Margulis-Soifer}
	H.~Abels, G.~Margulis and A.~Soifer, \emph{Semigroups containing proximal linear maps}, Israel J.~Math. \textbf{91} (1995), 1-30.
	
	
	\bibitem{Benoist}
	Y.~Benoist, \emph{Propri\'{e}t\'{e}s asymptotiques des groupes lin\'{e}aires}, Geom.~Func.~Anal. \textbf{7} (1997), 1-47.
	
		
	\bibitem{Benoist-Quint-linear}
	Y.~Benoist, J.F.~Quint, \emph{Central limit theorem for linear groups}, Ann.~Probab. \textbf{44} (2016), 1308-1340.
	
	
	\bibitem{Benoist-Quint}
	Y.~Benoist, J.F.~Quint, \emph{Central limit theorem on hyperbolic groups}, Izv.~Akad. Nauk Ser.~Mat. \textbf{80} (2016), 3-23.
	
	\bibitem{Bjorklund}
	M.~Bjorklund, \emph{Central limit theorem for Gromov hyperbolic groups}, J.~Theoret.~Probab. \textbf{23} (2010), 871-887.
	
	\bibitem{Boulanger-Mathieu-Sert-Sisto}
	A.~Boulanger, P.~Mathieu, C.~Sert and A.~Sisto, \emph{Large deviations for random walks on hyperbolic spaces}, {\texttt arXiv:2008.02709v1} (2020).
	
	\bibitem{Bourbaki-algebra}
	N.~Bourbaki.
	\newblock {\em \'{E}l\'{e}ments de math\'{e}matique, Alg\`{e}bre}.
	
	\bibitem{Brezis}
	H.~Brezis.
	\newblock {\em Functional analysis, Sobolev spaces and Partial Differential Equations}.
	\newblock Springer, New York, 2011.
	
	\bibitem{Bridson-Haefliger}
	M.R.~Bridson, A.~Haefliger.
	\newblock {\em Metric Spaces of Non-Positive Curvature}.
	\newblock Grundlehren der mathematischen Wissenschaften, Springer, Berlin, 1999.
	
	
	\bibitem{De-Ze}
	A.~Dembo, O.~Zeitouni.
	\newblock {\em Large deviations Techniques and Applications}.
	\newblock Stochastic Modelling and Applied Probability, Springer, Berlin, 2010.


	\bibitem{Dussaule}
	M.~Dussaule, \emph{Local limit theorems in relatively hyperbolic groups I: rough estimates}, Ergodic Theory Dynam.~Systems (2021), 1-41, {\texttt doi.10.1017/etds.2021.7}.
	
	\bibitem{Epstein}
	D.B.A~Epstein et al.
	\newblock {\em Word Processing in Groups}.
	\newblock Johns and Bartlett Publishers, Boston, MA, 1992.
	
	
	\bibitem{Ghys-deLaHarpe}
	E.~Ghys, P.~de la Harpe.
	\newblock {\em Sur les groupes hyperboliques d'apr\`{e}s Mikhael Gromov}.
	\newblock Progress in Mathematics, Birkha\"{u}ser, Boston, MA, 1990.
	
	\bibitem{Gouezel}
	S.~Gou\"{e}zel, \emph{Local limit theorem for symmetric random walks in Gromov-hyperbolic groups}, J.~Amer.~Math.~Soc. \textbf{27} (2014), 893-928.
	
	\bibitem{Gromov}
	M.~Gromov,
	\emph{Hyperbolic groups}, MSRI Publ. \textbf{8} (1987), 75-263.
	
	
	\bibitem{Guivarch}
	Y.~Guivarc'h,
	\emph{Sur la loi des grands nombres et le rayon spectral d'une marche al\'{e}atoire}, Ast\'{e}risque \textbf{74} (1980), 47-98.
	
	
	\bibitem{deLaHarpe}
	P.~de la Harpe.
	\newblock {\em Topics in Geometric Group Theory}.
	\newblock Chicago Lectures in Mathematics, The University of Chicago Press, Chicago, 2000.
	
	\bibitem{Kesten}
	H.~Kesten, \emph{Symmetric random walks on groups}, Trans.~Amer.~Math.~Soc. \textbf{92} (1959), 336-354.
	
	\bibitem{Kingman}
	J.F.C.~Kingman, \emph{The ergodic theory of subadditive processes}, J.~Roy.~Statist.~Soc.~Ser. B \textbf{30} (1968), 499-510.
	
	\bibitem{Klenke}
	A.~Klenke.
	\newblock {\em Probability theory. A Comprehensive Course}. Second Edition. 
	\newblock Universitext, Springer-Verlag, London, 2014.
	
	\bibitem{Lalley}
	S.P.~Lalley, \emph{Finite range random walks on free groups and homogeneous trees}, Ann.~Probab. \textbf{21} (1993), 2087-2130.
	
	\bibitem{Lanford}
	O.E.~Lanford,
	\emph{Entropy and equilibrium states in classical statistical mechanincs}, Statistical Mechanics and Mathematics Problems, Lecture Notes in Physics, Springer-Verlag, Berlin, 1973. 
	
	\bibitem{Ledrappier}
	F.~Ledrappier,
	\emph{Some asymptotic properties of random walks on free groups}, CRM Proceedings and Lecture Notes  \textbf{28} (2001), 117-152.
	
	\bibitem{Lyndon}
	R.C.~Lyndon, \emph{Grushko's theorem},
	Proc.~Amer.~Math.~Soc. \textbf{16} (1965), 822-826.
	
	\bibitem{Lyons-Peres}
	R.~Lyons, Y.~Peres.
	\newblock {\em Probability on Trees and Networks}.
	\newblock Cambridge Series in Statistical and Probabilistic Mathematics, Cambridge University Press, New York, 2016.
	
	\bibitem{Mathieu-Sisto}
	P.~Mathieu, A.~Sisto,
	\emph{Deviation inequalities for random walks},
	Duke Math.~J. \textbf{169} (2020), 961-1036.
	
	\bibitem{Mueller}
	S.~M\"{u}ller, \emph{Recurrence for branching Markov chains}, Electron.~Commun.~Probab. \textbf{13} (2008), 576-605.
	
	\bibitem{Ohshika}
	K.~Ohshika.
	\newblock {\em Discrete groups}.
	\newblock Translations of Mathematical Monographs, American Mathematical Society, Providence, RI, 2002.
	
	\bibitem{Polya}
	G.~Polya,
	\emph{\"{U}ber eine Aufgabe der Wahrscheinlichkeitstheorie betreffend die Irrfahrt im Stra{\ss}ennetz},
	Math.~Ann. \textbf{84} (1921), 149-160.
	
	\bibitem{Quint}
	J.F.~Quint, \emph{C\^{o}nes limites des sous-groupes discrets des groupes r\'{e}ductifs sur un corps local}, Transformation groups \textbf{7} (2002), 247-266.
	
	
	\bibitem{Sawyer-Steger}
	S.~Sawyer, T.~Steger,
	\emph{The rate of excape for anisotropic random walks in a tree},
	Probab.~Theory Related Fields \textbf{76} (1987), 207-230.
	
	\bibitem{Serre}
	J.P.~Serre.
	\newblock {\em Arbres, Amalgames, $\SL_2$}.
	Ast\'{e}risque, Soc.~Math.~France, 1977.
	
	\bibitem{Cagri-thesis}
	C.~Sert, \emph{Joint Spectrum and Large Deviation Principles for Random Matrix Products}, PhD Thesis, Universit\'{e} Paris-Sud, 2016. 
	
	\bibitem{Cagri-note}
	C.~Sert, \emph{Joint spectrum and large deviation principle for random matrix products}, C.~R.~Acad.~Sci.~Paris, Ser.~I \textbf{355} (2017), 718-722.
	
	\bibitem{Cagri-annals}
	C.~Sert, \emph{Large deviation principle for random matrix products}, Ann.~Probab. \textbf{47} (2019), 1335-1377.
	
	\bibitem{Woess}
	W.~Woess.
	\newblock {\em Random Walks on Infinite Graphs and Groups}.
	\newblock Cambridge Tracts in Mathematics, Cambridge University Press, Cambridge, 2000.
	
	\bibitem{Wolf}
	J.~Wolf, \emph{Growth of finitely generated solvable groups and curvature of Riemannian manifolds}, J.~Diff.~Geom. \textbf{2} (1968), 421-446.
	
\end{thebibliography}
\end{document}